\documentclass[11pt]{amsart}
\usepackage{amsmath,amssymb}
\usepackage{bm}
\usepackage{xcolor}
\usepackage{mathrsfs}
\usepackage{verbatim}
\usepackage{float}
\usepackage{setspace}
\usepackage{amsthm}
\usepackage{geometry}
\usepackage{color}
\usepackage{graphicx}
\usepackage{stmaryrd}
\usepackage{authblk}
\usepackage{etoolbox}
\usepackage{placeins}

\newtheorem{myTheo}{Theorem}[section]
\newtheorem{mylem}{Lemma}[section]

\newtheorem{remark}{Remark}[section]

\theoremstyle{definition}
\newtheorem{definition}{Definition}[section]
  \makeatletter
  \@addtoreset{equation}{section}
  \makeatother

\makeatletter
\let\maketitle\AB@maketitle
\let\authors\@author
\renewcommand{\@setauthors}{%
  \begin{center}
    \@author
  \end{center}
}

\makeatletter
\patchcmd{\abstract}
  {\item[\hskip\labelsep\scshape\abstractname.]}
  {\item[]\noindent
   \makebox[\linewidth][c]{%
     \normalfont\normalsize\bfseries\abstractname}
   \par\medskip}
  {}{}
\makeatother
\begin{document}

\title[Morley element for SGE problem]{Morley element for the linear strain gradient elasticity problem}

\author[a]{Xiao Li}
\author[a,*]{Shudan Tian}
\author[a]{Hua Wang}

\affil[a]{Hunan Key Laboratory for Computation and Simulation in
Science and Engineering, National Center for Applied Mathematics
in Hunan, Xiangtan University, Xiangtan 411105, Hunan, China}


\date{}

\thanks{\textsuperscript{*}Corresponding author.
E-mail: \texttt{shudan.tian@xtu.edu.cn}.}

\thanks{E-mail addresses:
\texttt{202521511176@smail.xtu.edu.cn(Xiao Li)},
\texttt{wanghua@xtu.edu.cn} (Hua Wang).}

\thanks{The work of ST was supported by the NSFC
under Grant No. 12401483.}


\begin{abstract}
We establish a discrete $H^1$-Korn inequality for finite element
spaces satisfying a second order weak continuity condition.
This implies that classical $H^2$-nonconforming elements
satisfy the discrete Korn inequality.
Based on this result, we develop a stabilized Morley finite element
method for linear strain gradient elasticity using only piecewise
quadratic polynomials. Thanks to the discrete $H^1$-Korn inequality,
this method does not require any penalty parameter.
We prove that the method is robust with respect to both $\lambda$
and $\iota$. Numerical experiments confirm the theoretical results.

\end{abstract}



\maketitle

\def\a#1{\begin{align*}#1\end{align*}}\def\an#1{\begin{align}#1\end{align}}
\def\ad#1{\begin{aligned}#1\end{aligned}}

\section{Introduction}
We consider the following linear strain gradient elasticity(SGE) problem in pure displacement form:
\begin{equation}
	\begin{cases}
	\iota^2 \mathrm{div}  \Delta \sigma( u)-\mathrm{div}\sigma( u)= f\quad & \text{in }\Omega, \\
	 u = \frac{\partial  u}{\partial \mathbf{n}} =  0 \quad &\text{on }\partial\Omega,
	\end{cases}
	\label{LS}
\end{equation}
where $\Omega$ is a bounded d-dimensional domain, and $\mathbf{n}$ is the unit outward normal to $\partial \Omega$. $ u$ is a displacement field in $\mathbb{R}^d$, and $\sigma( u)=2\mu \epsilon( u)+\lambda(\mathrm{tr})(\epsilon( u))I$ denotes the stress tensor. Here $\mu\in (\mu_1,\mu_2)$ and $\lambda\in (0,+\infty)$ are Lamé constants, and $I$ is the identity matrix. The strain tensor is defined by $\epsilon(u)=\frac{1}{2}(\nabla u+(\nabla  u)^T)$. Problem \eqref{LS} can be viewed as elastic type of a fourth order singular perturbation problem. Comparing with the standard fourth order singular perturbation problem, \eqref{LS} have to consider the effect by the parameter $\lambda$. Our goal is to construct a displacement method that is robust with respect to both $\lambda$ and $\iota$ using only piecewise quadratic polynomial. 

\subsection{Linear strain gradient model}

Problem~\eqref{LS} was proposed by Aifantis
\cite{Ru1993A,Aifantis2011overview} as a simplified version of Mindlin's
strain-gradient elasticity theory \cite{MINDLIN1964,MINDLIN1965}. In contrast
to the general Mindlin theory, which involves several higher-order material
parameters, this model introduces only a single internal length-scale parameter.
Despite its simple constitutive structure, it effectively captures
size-dependent mechanical behavior at the micro- and nanoscale. In particular,
the higher-order regularization can remove the stress singularities predicted
by classical elasticity near crack tips, dislocation cores,
and concentrated loads, leading to bounded and smoother stress fields. 

\subsection{Numerical methods for the SGE model}
The main difficulty in designing a method robust
with respect to both $\iota$ and $\lambda$ for problem \eqref{LS} is that it must provide convergent approximations in both the second order and fourth order
terms, while the images of the finite element space under
$\operatorname{div}$ and $\nabla\operatorname{div}$ have suitable approximation properties. The use of $H^2$-conforming finite elements is
a natural choice. However, the numerical experiments reported in
\cite{TianThesis2021} exhibit a loss of convergence order as the material approaches
the incompressible limit. The $C^0$ interior penalty discontinuous
Galerkin ($C^0$-IPDG) method is also effective for singularly perturbed
problems \cite{BrennerSung2005}, but its may a not a $\lambda$-robust scheme.
Indeed, the $\nabla_h\operatorname{div}_h$ of $P_2$ Lagrange space does not map onto the full space of elementwise constant vector fields. Moreover, the standard $P_2$ Lagrange displacement method for
linear elasticity is not $\lambda$-robust in general \cite{AinsworthParker2022}. 

The locking-free discontinuous Galerkin(DG) method for linear elasticity was proposed in \cite{HansboLarson2002}. However, a fully DG formulation of the SGE problem involves additional trace terms, and requires $\lambda$-robust estimates for $\nabla\mathrm{div}$. Therefore, the analysis in \cite{HansboLarson2002} cannot be directly extended to the SGE problem.

Numerical methods for problem \eqref{LS} were introduced rather early in the
engineering literature, including conforming discretizations \cite{C1forLSG2009,C1forLSG2010,GAforLSG2011}
and two-step approaches \cite{Ru1993A,AskesMorataAifantis2008}. However, mathematic well-posedness and convergence
analyses of these methods are mentioned few, specially for the parameter-robust analysis. A mathematical analysis of an SGE discretization was presented in \cite{Ming2017}. Li, Ming, Shi established the well-posedness of weak formulation for problem \eqref{LS} and proposed an
$H^2$-nonconforming finite element method. Compared with
conforming methods, this approach avoids the difficulties in imposing
boundary conditions caused by the high continuity requirements at vertices and
involves fewer degrees of freedom. Nevertheless, the
method was shown to be robust only with respect to the length
parameter $\iota$.

Methods robust with respect to both $\iota$ and the Lamé parameter
$\lambda$ were developed in \cite{MingMixLSG,HuangLSGSIAM}. These methods treat
$\operatorname{div} u$ as an additional unknown and discretize
the resulting formulation using mixed finite elements. By employing
reduced integration, \cite{TianThesis2021} constructed another class of
parameter-robust nonconforming finite element methods in two dimensions. 
In \cite{HuangHuangTang2025}, a bubble-enriched $H^2$-nonconforming
displacement space was constructed whose divergence coincides with
an $H^1$-nonconforming scalar space. This compatibility permits an
equivalent displacement--pressure mixed formulation without
projecting the discrete divergence, which yields a method robust with respect to both $\lambda$
and $\iota$ without reduced integration. More recently, \cite{HuangTang2026} proposed another robust
mixed finite element method by directly discretizing
the full tensor field
$\nabla\varepsilon( u)$. Virtual element methods for problem \eqref{LS} on polygonal meshes were proposed
in \cite{HuangYu2026}.

\subsection{Main contributions}
The main contributions of this paper are a parameter-robust(for both $\iota$ and $\lambda$) discretization
of problem~\eqref{LS} based on the Morley element and the establishment of a discrete Korn inequality for finite element spaces satisfying a second order weak continuity condition, which implies that classical $H^2$-nonconforming elements satisfy the discrete Korn inequality. Unlike the quadratic
Lagrange element, the Morley element is a piecewise quadratic nonconforming finite
element that can be used to discretize the biharmonic equation directly.
In $d$ dimensions, its degrees of freedom consist of the integral averages
of the function over the $(d-2)$-dimensional subsimplices and the integral
averages of its normal derivative over the $(d-1)$-dimensional faces.
The canonical Morley interpolation operator satisfies commuting
properties~\cite{Brenner2015CR}, which yield approximation estimates
for the discrete divergence $\operatorname{div}_h$ and its broken
gradient $\nabla_h\operatorname{div}_h$. 

However, the Morley element does not yield a convergent
discretization of the Poisson equation and therefore cannot provide an $\iota$-robust scheme when applied
directly to the SGE model~\eqref{LS}. A modified Morley method was proposed in \cite{MingLS1},
where a continuous piecewise linear interpolation is employed in the
second order term. This method is not robust with respect to $\lambda$, because $P_1$ Lagrange element has the locking phenomena \cite{BabuskaSuri1992}. The locking-free analysis in~\cite{HansboLarson2002} relies on BDM interpolation. Since BDM element space generally do not belong to the Morley finite element space, the same approach cannot be directly applied to the second order elasticity term.

Consequently, the main difficulty in constructing a parameter-robust
Morley-based method lies in the treatment of the second order elasticity
term. Although the trace consistency terms generated by elementwise
integration by parts enable the Morley element to discretize the linear
elasticity operator, a direct application of the standard discontinuous Galerkin formulation leads to the following term
\[
  \lambda
  \sum_{F\in\mathcal F_h^i}
  \left\langle
    \{\operatorname{div}_h\boldsymbol v_h\},
    [\boldsymbol u-\Pi_M\boldsymbol u] 
  \right\rangle_F .
\]
Since the Morley degrees of freedom contain no moments of the function
values over the $(d-1)$-dimensional faces, a
direct estimate of this term is not uniform with respect to $\lambda$.
Motivated by \cite{YiLeeZikatanov2022}, we add a stabilization term to
the Morley discretization. However, the analysis for the enriched piecewise linear space does not extend directly to the piecewise quadratic Morley space. Thus, we modify the above term as follows:
\[
  \lambda
  \sum_{F\in\mathcal F_h^i}
  \left\langle
    \{P_T^0\operatorname{div}_h\boldsymbol v_h\},
    [\boldsymbol u-\Pi_M\boldsymbol u] 
  \right\rangle_F .
\]
Together with the commuting properties of the canonical Morley
interpolation operator, this modification yields a scheme robust
with respect to both $\iota$ and $\lambda$. In particular, these
commuting properties allow us to retain the original Morley space
without any additional enrichment.

Since our modification only on the second order term,
the fourth order term still is the standard Morley discretization.
Adapting the analysis in~\cite{HuZhangSPP}, we derive error
estimates under the natural energy regularity
$u\in H_0^2(\Omega;\mathbb R^d)$.
This low-regularity analysis is particularly useful for obtaining
estimates uniform in $\iota$, since higher order Sobolev norms
of the solution may grow as $\iota\to0$.

We also analyze a superpenalty Morley element method based
on~\cite{SPMWX}. The commuting properties of the canonical
Morley interpolation operator allow us to establish robustness
with respect to $\lambda$ without adding any projection
into the scheme. However, the superpenalty method requires an a priori choice
of the penalty exponent $p$. The resulting error bounds
exhibit a trade-off between higher convergence rates for
smooth solutions and convergence rates that are uniform
with respect to $\iota$.
\subsection{Korn inequality}
To establish the coercivity of the discrete bilinear form, we require
both a broken $H^2$--Korn inequality \cite{MingH2korn} for the fourth-order term and a
discrete $H^1$--Korn inequality for the second-order term. As shown in~\cite{MingH2korn}, the strain gradient controls the Hessian pointwise, which directly yields the required coercivity
estimate for the fourth-order term. Therefore, the key step in
establishing the coercivity of the discrete bilinear form is to
prove a discrete $H^1$-Korn inequality.

Discrete Korn inequalities for nonconforming finite element
spaces have been studied in various settings; see,
e.g.,~\cite{BrennerKorn2004,WilliamsHong2024,Wang1994Korn}.
In particular, the results in~\cite{BrennerKorn2004,Wang1994Korn}
imply that the two-dimensional Morley space satisfies the
discrete Korn inequality. To the best of our knowledge,
an analogous result for the three-dimensional Morley space
has not previously been established.

In this paper, we prove a discrete $H^1$--Korn inequality for every
finite element space satisfying the second-order weak continuity
condition. The proof combines the
continuous $H^2$ Korn inequality with local inverse estimates. This implies that the Morley element and other standard $H^2$-nonconforming elements satisfy the $H^1$-Korn inequality. Thanks to this property, no additional penalty
parameters are required for coercivity.

The paper is organized as follows.
Section~2 introduces the notation, the weak formulation of
the SGE model, and the relevant regularity estimates.
Section~3 presents the stabilized Morley method and establishes
the discrete Korn inequality and interpolation estimates.
Section~4 provides the a priori error analysis of the proposed
method. Section~5 introduces and analyzes the superpenalty
Morley method. Finally, Section~6 presents numerical experiments
illustrating the theoretical results.

\section{Preliminaries}
Some notations will be introduced in this section, which is used throughout this article. The inner product in the Lebesgue space $L^2(G)$ with respect to $G\subseteq \Omega$ is denoted by $(\cdot,\cdot)_{0,G}$ with the associated $L^2$ norm $\|\cdot\|^2_{0,G}=(\cdot,\cdot)_{0,G}.$  Standard Sobolev space \cite{Adams} on $G$ denote as $H^m(G)$, and define 
$(\cdot,\cdot)_{m,G}=(\nabla^m \cdot,\nabla^m\cdot)_{0,G}$. Let $|\cdot|^2_{m,G}=(\cdot,\cdot)_{m,G}$ be a semi-norm on $H^m(G),$ then standard norm  on $H^m(G)$ is $\|\cdot\|_{m,G}=\sum_{j=0}^{m}|\cdot|_{j,G}.$ If $G=\Omega$, we simply write 
$(\cdot,\cdot)_m$, $|\cdot|_m$ and $\|\cdot\|_m.$  $H_0^2(\Omega)$ represents the $H^2(\Omega)$ space satisfying clamped boundary condition and $H_0^1(\Omega)$ space represents $H^1(\Omega)$ space with homogeneous Dirichlet boundary condition. \\ \qquad

To give a robust estimate for $\iota$, we introduce fractional order Sobolev space first.  For a non-integer $s>0$, $H^s(\Omega)$ is a fractional order Sobolev space. Let $m$ be the largest integer less than $s$, i.e., $s-1< m<s$, then define the semi-norm $|\cdot|_s$ of $v$ as
$$
|v|^2_s = \sum_{|\alpha|=m}\int_{\Omega}\int_{\Omega}\frac{|\partial^{\alpha}v(\pmb x)-\partial^{\alpha}v(\pmb y)|^2}{|\pmb x -\pmb y|^{d+2(s-m)}}\mathrm{d}\pmb x\mathrm{d}\pmb y,
$$
and the corresponding norm is defined as 
$$
\|\cdot\|_s^2=\|\cdot\|_m^2+|\cdot|_s^2.
$$
Actually, $H^s(\Omega)$ can equivalently be obtained by interpolation between $H^m(\Omega)$ and $H^{m+1}(\Omega)$, namely
$$
H^s(\Omega)=[H^{m+1}(\Omega),H^{m}(\Omega)]_{\theta},
$$ 
where $\theta = m+1-s.$ The following inequality will helpful obtain $\iota$-robust estimation.
\begin{equation}
	\|v\|_{s}\leq C\|v\|_{m+1}^{1-\theta}\|v\|_m^{\theta}.
    \label{eq:sobolev-interpolation}
\end{equation} 
Here the constant $C$ depends on $\Omega$ and $s$.
Let $\mathcal T_h$ be a shape regular simplicial mesh of
$\Omega\subset\mathbb R^d$ and $h_T$ denote the diameter of $T$,~$h=\max_{T\in\mathcal{T}_h}\{h_T\}$.
For $j=1,\ldots,d$, let $\mathcal F^{d-j}(\mathcal T_h)$
denote the set of all $(d-j)$-dimensional faces of the mesh.
In particular, $\mathcal F^{d-2}(\mathcal T_h)$ consists of
vertices in two dimensions and edges in three dimensions.
A face $F\in\mathcal F^{d-j}(\mathcal T_h)$ is called a
boundary face if $F\subset\partial\Omega$.
We denote the set of such faces by
$\mathcal F_b^{d-j}(\mathcal T_h)$ and define the set of
interior faces by
\[
\mathcal F_i^{d-j}(\mathcal T_h)
:=\mathcal F^{d-j}(\mathcal T_h)
  \setminus\mathcal F_b^{d-j}(\mathcal T_h).
\]
For simplicity, we write
\[
\mathcal F_h:=\mathcal F^{d-1}(\mathcal T_h),
\qquad
\mathcal F_h^i:=\mathcal F_i^{d-1}(\mathcal T_h),
\qquad
\mathcal F_h^b:=\mathcal F_b^{d-1}(\mathcal T_h).
\]
For each element or face $G$ of the mesh, let
$h_G:=\operatorname{diam}(G)$, and set
$h:=\max_{T\in\mathcal T_h}h_T$.
Let 
$$H^m(\mathcal{T}_h)=\{v\in L^2(\mathcal{T}_h)|~v|_T\in H^m(T),~\forall T\in\mathcal{T}_h\}.$$
Then we define the broken inner product, norm and seminorm by 
\[ 
(v,w)_{\mathcal{T}_h}:=\sum_{T\in\mathcal{T}_h}(v,w)_{0,T}, 
\qquad 
\|v\|_{H^s(\mathcal{T}_h)}^2:=\sum_{T\in\mathcal{T}_h}\|v\|_{s,T}^2, 
\] 
\[ 
|v|_{H^s(\mathcal{T}_h)}^2:=\sum_{T\in\mathcal{T}_h}|v|_{s,T}^2, 
\qquad 
\|v\|_{L^2(\mathcal{T}_h)}:=\|v\|_{H^0(\mathcal{T}_h)}. 
\] 
For integer $m$, this gives 
\[ 
(\nabla_h^m v,\nabla_h^m v)_{\mathcal{T}_h}=\sum_{T\in\mathcal{T}_h}(\nabla^m v,\nabla^m v)_{0,T} 
=|v|_{H^m(\mathcal{T}_h)}^2. 
\] 
The operators $\nabla_h$, $\operatorname{div}_h$, and $D_h^2$
denote the broken gradient, divergence, and Hessian, respectively. For example,
\[
(\nabla_h v)|_T := \nabla(v|_T),
\qquad \forall T\in\mathcal T_h,\quad
\forall v\in H^1(\mathcal T_h).
\]
For a piecewise $H^1$ vector field $v$, we write 
\[ 
\epsilon_h(v):=\tfrac12\bigl(\nabla_h v+(\nabla_h v)^T\bigr), 
\qquad 
\sigma(v):=2\mu\epsilon_h(v)+\lambda\operatorname{div}_h v\,I. 
\] 
Similarly 
$
(\cdot,\cdot)_{\mathcal{F}_h}= \sum_{F\in\mathcal{F}_h}(\cdot,\cdot)_{0,F}. 
$
For $F\in \mathcal{F}^i_h$, let $T^+$ and $T^-$ be two elements sharing $F$. Next we will define the jump and average operators for functions in $H^{s}(\mathcal{T}_h),~s>1/2.$ 
The average operator is defined by 
$$
\{v\}|_F=\frac{1}{2}(v|_{T^+\cap F}+v|_{T^-\cap F}),
$$ 
where $v$ can be a scalar, vector or matrix-valued function. Let $\pmb n_{\pm}$ be an out unit normal vector for $F\cap T_{\pm}.$ Write $v^\pm:=v|_{T^\pm}$ for the two traces. For a scalar function $v$, we define its jump across $F$ by 
$$
[v]|_F=v^+\pmb n_+ + v^-\pmb n_-, 
$$
which is vector-valued.
For a vector function $v$, the jump operator is defined as
$$
[v]|_F=v^+\cdot \pmb n_+ + v^-\cdot \pmb n_-, \text{ and } \llbracket v \rrbracket|_F=v^+\otimes \pmb n_+ + v^-\otimes \pmb n_-. 
$$
Here $\otimes$ denotes the outer product, and the $[v]$ will be a scalar and $\llbracket v \rrbracket$ be a matrix.
If $v$ is a matrix function, define
$$
[v]|_F=v^+ \pmb n_+ + v^- \pmb n_-, \text{ and } \llbracket v \rrbracket|_F=v^+\otimes \pmb n_+ + v^-\otimes \pmb n_-, 
$$
where $[v]$ is vector-valued and $\llbracket v\rrbracket$ is a third-order tensor. 
For a third-order tensor $G$, the normal jump contracts its last index: 
\[ 
[G]_{ij}|_F:=\sum_{k=1}^d 
\bigl(G_{ijk}|_{T^+}(\pmb n_+)_k+G_{ijk}|_{T^-}(\pmb n_-)_k\bigr). 
\] 
On a boundary face, the average equals the interior trace, and the jumps 
are given by the same formulas with the exterior trace set to zero and 
the interior normal equal to the outward normal to $\partial\Omega$. 
We denote the outward unit normal to $\partial T$ by $\pmb n_T$. 
The notation $a\lesssim b$ means $a$ can be controlled by $Cb$, where the constant $C$ is independent of mesh-size and $\iota,\lambda$ but may depend on $\mu$. 

The weak formulation of problem \eqref{LS} is to find $u\in H_0^2(\Omega;\mathbb{ R}^d)$ such that
\begin{equation}
    \iota^2(\nabla\sigma(u),\nabla \epsilon(v)) + (\sigma(u),\epsilon(v)) = (f,v),~\forall v\in H_0^2(\Omega;\mathbb{R}^d). 
    \label{eq:LSG-cweak}
\end{equation}
The following $H^2$-Korn inequality implies the coercivity.
\begin{mylem}[\cite{MingH2korn}]
For any $v\in (H^2(\Omega))^d$, $d=2,3$, there holds
\begin{equation}
    |\nabla\epsilon(v)|^2 \geq C(d)|D^2v|^2
    \qquad \text{a.e. in }\Omega,
    \label{eq:H2-korn}
\end{equation}
where $|\cdot|$ denotes the Euclidean norm of a tensor
and $C(d)>0$ depends only on $d$.
\end{mylem}

Consider the limiting case of problem \eqref{eq:LSG-cweak}, i.e., the linear elasticity problem
\begin{equation}
	\begin{cases}
&		-\mathrm{div}\sigma(u_0)=f,~\text{in } \Omega\\
&u_0=0,~\text{on } \partial\Omega.
	\end{cases}
\label{Eq:Elasticity}
\end{equation}
We assume that
\begin{equation}
	\|\sigma(u_0)\|_{1} 
	\lesssim
	\|f\|_{0}. 
    \label{eq:regularity-elastic}
\end{equation}
Regularity estimate \eqref{eq:regularity-elastic} is established when $\Omega$ is convex in 2D \cite{BrennerSung1992}.
For problem \eqref{eq:LSG-cweak}, we assume that
the solutions $u$ and $u_0$ satisfy
\begin{align}
&\lambda\|\mathrm{div}(u-u_0)\|_0+|u-u_0|_1
\lesssim \iota^{1/2}\|f\|_0,
\label{Reg-u0}\\
&\sum_{i=1}^2\iota^i
\bigl(|u|_{i+1}+\lambda|\mathrm{div}u|_i\bigr)
\lesssim \iota^{1/2}\|f\|_0,
\label{Reg}
\end{align}
with hidden constants independent of $\lambda$ and $\iota$. 
For sufficient conditions ensuring these estimates \cite{HuangLSGSIAM,MingMixLSG}. Combining \eqref{eq:sobolev-interpolation}, gives
\[
\|u\|_{3/2}+\|\mathrm{div }u\|_{1/2}\lesssim \|f\|_0.
\]
\section{Finite element discretization}
This section will introduce a robust finite element method for problem \eqref{eq:LSG-cweak} based on the Morley element\cite{Morley}. The local shape function space is $P_2(T)$, and the degrees of freedom are defined by
\begin{equation}
	\begin{cases}
&		\int_{e}v\mathrm{d}l,~\text{for all }e\in \mathcal{F}^{d-2}(T) ,\\
&		\int_F \partial_n v\mathrm{d}s, \text{for all }F\in \mathcal{F}^{d-1}(T).
	\end{cases}
    \label{eq:dofMorley}
\end{equation}
The associated Morley element space $V_M$ is
\begin{align*}
V_M=\bigl\{
v\in (L^2(\Omega))^d:\;&
v|_T\in (P_2(T))^d,
\quad \forall T\in\mathcal{T}_h,
\int_F
\bigl[\nabla_h v\bigr]\,ds=0, 
\quad \forall F\in\mathcal{F}_i^{d-1}(\mathcal{T}_h),\\
&
\int_e v|_T\,ds=\int_e v|_{T'}\,ds,
\quad \forall e\in\mathcal{F}^{d-2}(\mathcal{T}_h),
\quad \forall T,T'\in\omega_e
\bigr\},
\end{align*}
where $\omega_e$ is a star patch of $(d-2)$-dimensional face $e$, i.e., 
\[
\omega_e
:=\{T\in\mathcal{T}_h:\ e\subset\overline{T}\}.
\]
When $d=2$, Morley element will continuous on each interior vertices. And $V_M^0$ denotes the Morley element space with vanish boundary degrees of freedom. 
It should be noted that $\int_F[v_h]\mathrm{d}s$ will not vanish in the Morley element space. This implies that Morley elements may not converge for solving second order elliptic equations. For this reason, we propose the following finite element discretization for problem \eqref{eq:LSG-cweak}: Find $ u_h \in V^0_M$ such that 
   
   \begin{equation}
   	\mathcal{A}( u_h, v_h) =( f, v_h)_{\mathcal{T}_h} ,\text{ for all } v_h \in V^0_M, 
   	\label{eq:LSG-weak}
   \end{equation}
   where 
   $$\begin{aligned}
&   \mathcal{A}=a_{\iota}(u_h,v_h)+J(v_h,u_h)-J(u_h,v_h)+s(u_h,v_h),\\
  & 	a_\iota( u_h, v_h) = \iota^2 (\nabla_h \sigma( u_h), \nabla_h\epsilon_h( v_h))_{\mathcal{T}_h}+  (\sigma( u_h), \epsilon_h( v_h))_{\mathcal{T}_h}, \\ 
   	&J(v_h,u_h)= ( \{P_T^0\sigma( v_h)\},  \llbracket u_h\rrbracket)_{\mathcal{F}_h},\\& s(u_h,v_h)= \lambda^2(h_F[P_T^0\operatorname{div}_h u_h],[P_T^0\operatorname{div}_h v_h])_{\mathcal{F}^i_h} .  
   \end{aligned}$$ 
Here $P_T^0$ denotes the $L^2$-orthogonal projection onto
constant on $T$, applied componentwise to vector-
and tensor-valued functions, moreover we have 
\[
P_T^0\epsilon_h(u_h)=\epsilon_h(u_h)|_{x_T},~P_T^0\operatorname{div}_h u_h=\operatorname{div}_h u_h|_{x_T},~\forall u_h\in V_M. 
\]
where $x_T$ is the center point of $T$. The projection of the volumetric term plays an important role
in the $\lambda$-robust error analysis. Actually, one alternative formulation is that we only project the divergence term instead of the full $\sigma(u)$ in $J(v_h,u_h)$, i.e.,     
  \begin{equation} 
  	J'(v_h,u_h)=  (\{2\mu\epsilon_h(v_h)+\lambda P_T^0\operatorname{div}_h  v_h\mathrm{I}\},\llbracket u_h\rrbracket)_{\mathcal{F}_h}. 
  	\label{LSG-FEM'}
  \end{equation} 
The non-symmetric part leads the coercivity of \eqref{eq:LSG-weak} independent of $\lambda$. The stabilization term is inspired by \cite{YiLeeZikatanov2022}, and the projection in the stabilization term leads the inconsistency.    
\subsection{Well-posedness}
Define the following $\iota$-related norm
$$
\interleave  v \interleave _{\iota,h}^2 := \sum_{T \in \mathcal{T}_h} \Big( \iota^2 | v |^2_{2,T} +| v |^2_{1,T} \Big).
$$
The coercivity can be directly from the Morley element satisfies the $H^1$ and $H^2$ Korn inequality, even though it is not a continuous element. In fact, we can prove all finite elements with second order weak continuity satisfy Korn inequality. 
\begin{definition}[Second order weak continuity]
	A finite element space $V_h$ is said to
	satisfy the \emph{second order weak continuity property} if, for
	every $v_h\in V_h$, every $F=T^+\cap T^-\in\mathcal F_h^i$, and
	every unit vector $\gamma\in\mathbb R^{d}$, there exist points
	$x_F,x_{F,\gamma}\in F$ such that
	\[
		v_h^+(x_F)=v_h^-(x_F),
		\qquad
		\partial_\gamma v_h^+(x_{F,\gamma})
		=\partial_\gamma v_h^-(x_{F,\gamma}).
	\]
	For a boundary face $F\in\mathcal F_h^b$, the corresponding
	homogeneous conditions are
	\[
		v_h(x_F)=0,
		\qquad
		\partial_\gamma v_h(x_{F,\gamma})=0.
	\]
\end{definition}
Most classical $H^2$-nonconforming elements satisfy the
second order weak continuity property.
In our definition, this property includes the corresponding
weak homogeneous boundary conditions and therefore ensures
that the broken $H^2$ seminorm is a norm on $V_{NC}^0$.
Let $I_A$ be the projection-averaging interpolation operator \cite{ShiWang2013FEM} \[I_A:V_{NC}^0\rightarrow V_A\cap H_0^2(\Omega),\] where $V_A$ is the finite element space of the Argyris element \cite{Argyris,Zhang2009C1}. Such operators are also referred to as enriching operators
or conforming companion operators \cite{MallikNataraj2016,Gallistl2015Morley}.
\begin{mylem}[\cite{ShiWang2013FEM}] 
\label{Lem:Enrich-error}
Let $V_{NC}^0$ be a finite element space satisfying the second order weak continuity property. 
	Then for all $v_h\in V_{NC}^0$ and 
$0\leq j\leq i\leq 2$,
\begin{equation}
|v_h-I_Av_h|_{H^j(\mathcal T_h)}^2
\lesssim
\sum_{T\in\mathcal T_h}
h_T^{2(i-j)}|v_h|_{H^i(T)}^2.
\label{eq:Enrich-error}
\end{equation}
\end{mylem}

\begin{myTheo}[Korn inequality for $H^2$-nonconforming FE]
	If $V_{NC}^0$ satisfies the $2$nd order weak continuity property, then 
	\begin{equation}
	|v_h|_{H^1(\mathcal{T}_h)}	\lesssim \|\epsilon_h(v_h)\|_{L^2(\mathcal{T}_h)}. 
	\end{equation}
\label{NC-H2-Korn}
\end{myTheo}
\begin{proof}
For all $v_h\in V_{NC}^0$ we have
	\begin{equation} 
	|v_h|_{H^1(\mathcal{T}_h)}\leq |v_h-I_Av_h|_{H^1(\mathcal{T}_h)}+|I_Av_h|_1. 
	\end{equation}
Noting that $I_Av_h\in H_0^2(\Omega),$ then 
\[
|I_Av_h|_1\lesssim \|\epsilon(I_Av_h)\|_0\leq \|\epsilon(v_h)\|_0+\|\epsilon(I_Av_h-v_h)\|_0 . 
\]
Inverse inequality, $H^2$ Korn inequality and stability \eqref{eq:Enrich-error} imply
\[
\begin{aligned}
\|\epsilon_h(I_Av_h-v_h)\|^2_0
\le |I_Av_h-v_h|^2_{1,h}\lesssim
\sum_{T\in\mathcal T_h}
h_T^2|v_h|_{2,T}^2\lesssim
\sum_{T\in\mathcal T_h}
h_T^2|\epsilon_h(v_h)|_{1,T}^2\lesssim \|\epsilon_h(v_h)\|^2_0.
\end{aligned}
\]

Thus, we finish the proof.
\end{proof}
\begin{mylem}[Coercivity]
	For any $v_h\in V_M^0$
	\begin{equation}
		\interleave  v_h \interleave _{\iota,h}^2+\sum_{F \in \mathcal{F}^i_{h}}h_F\lambda^2\|[P_T^0\operatorname{div}_h v_h]\|^2_{0,F} \lesssim \mathcal{A}(v_h,v_h). 
	\end{equation}
Here the hidden constant is independent of $\lambda$ and $\iota$, but may depend on $\mu$ and shape regular constant of the mesh.  
\label{lem:covercivity}
\end{mylem}

\subsection{Interpolation estimates}
In this section we will give the error analysis of \eqref{eq:LSG-weak}. Define
the interpolation operator of Morley element $\Pi_M : H^2(\Omega) \longrightarrow V_M$ and
$\Pi_Mu|_T \in  P_2(T)$ is
$$\left\{\begin{aligned}\int_e\Pi_M u|_T\mathrm{d}s&= \int_eu|_T\mathrm{d}s,~e\in \mathcal{F}^{d-2}(T),\\
	\int_F\partial_n\Pi_Mu|_T dS&=\int_F\partial_nu|_T\mathrm{d}S,~\forall F\in \mathcal{F}^{d-1}(T).\end{aligned}\right.$$
For vector-valued functions, the interpolation operator
$\Pi_M$ is defined componentwise.  The following commuting property can be found in \cite{Brenner2015CR}.
\begin{mylem} 
	For any $ v \in H^2(\Omega)$, then following commutativity property hold true:
	\begin{equation}
		D_h^2 \Pi_M{ v}=P_T^0 D^2 v,\quad 
		\nabla_h\Pi_M  u= \Pi _{cr}\nabla   u. 
		\label{eq:commuting}
	\end{equation}
	Here $\Pi_{cr}$ is the interpolation operator of Crouzeix-Raviart element \cite{CrouzeixRaviart1973}, defined on each $T\in\mathcal T_h$ by $(\Pi_{cr}v)|_T\in P_1(T)$ and 
\[
\int_F(\Pi_{cr}v)|_T\,ds=\int_F v\,ds,
\quad \forall F\subset\partial T.
\]
\end{mylem}
According to \eqref{Reg}, $\|u\|_m\lesssim\iota^{3/2-m}\|f\|_0,~m=2,3$, so that $\iota$-robust estimation relies on the fractional  interpolation theorem. The Sobolev embedding theorem implies that $\Pi_M$ is well defined on $H^s(\Omega)$ for $s > 3/2$. Following similar scaling argument in \cite{DupontScott1980}, we can obtain a fractional Sobolev estimate for $\Pi_M$.
\begin{mylem}
	Let $u\in H^{s}(\Omega),~3/2<s\leq 3.$  Then the following estimates hold. 
	\begin{align}
		&h^t\|u-\Pi_Mu\|_{H^t(\mathcal{T}_h)}\lesssim h^s\|u\|_{H^{s}(\mathcal{T}_h)},\\&\|\nabla u-\nabla_h\Pi_Mu\|_{L^2(\mathcal{T}_h)}=\|\nabla u-\Pi_{cr}\nabla u\|_{L^2(\mathcal{T}_h)}\lesssim h^{s-1}\|u\|_{H^s(\mathcal{T}_h)}. 
	\end{align}	
    Here $0\leq t\leq s$.
    \label{lem:interpolation1}
\end{mylem}
Since the Crouzeix--Raviart interpolation operator is not
well defined as a bounded operator on $H^{1/2}$, the limiting case
$s=3/2$ is unavailable in Lemma~\ref{lem:interpolation1}. However, for the exact solution $u\in H_0^2(\Omega)$ of
problem~\eqref{eq:LSG-cweak}, we have the following
parameter-robust interpolation estimate.
\begin{mylem}
	Let $u\in H^s(\Omega;\mathbb R^d)\cap H_0^2(\Omega;\mathbb R^d),~2\leq s\leq 3$ be the exact solution of problem \eqref{eq:LSG-cweak}, and $u$ satisfies \eqref{Reg-u0} and \eqref{Reg} then 
	\begin{align}
		&\|\sigma(u-\Pi_Mu)\|_{L^2(\mathcal{T}_h)}\lesssim	\min\{h^{1/2}\|f\|_0,h^{s-1}\|\sigma(u)\|_{s-1}\},\\
		&\iota\|\nabla_h\sigma(u-\Pi_Mu)\|_{L^2(\mathcal{T}_h)}\lesssim\min\{\iota^{1/2}\|f\|_0,\iota h^{s-2}\|\sigma(u)\|_{s-1}\}\label{eq:gradient-inter}. 
	\end{align}
	Here the hidden constant is independent of $\iota$ and $\lambda.$ But $\|\sigma(u)\|_{s-1}$ may depend on $\iota.$ 
	\label{lem:interpolation}
\end{mylem}
\begin{proof}
By the commutativity property \eqref{eq:commuting},
\[
\|\sigma(u-\Pi_Mu)\|_{L^2(\mathcal{T}_h)} =  \|\sigma(u)-\Pi_{cr}\sigma(u)\|_{L^2(\mathcal{T}_h)}.
\]
When $h\leq \iota$, Lemma \ref{lem:interpolation1} implies
\begin{equation}
    \|\sigma(u-\Pi_Mu)\|_{L^2(\mathcal{T}_h)} \lesssim h^{s-1}\|\sigma(u)\|_{s-1}\lesssim h^{s-1}\iota^{3/2-s}\|f\|_0\leq h^{1/2}\|f\|_0.
\end{equation}
For $h>\iota$, 
\[
\begin{aligned}
\sigma(u)-\Pi_{cr}\sigma(u)
={}&\sigma(u-u_0)
+\sigma(u_0)-\Pi_{cr}\sigma(u_0)
-\Pi_{cr}\sigma(u-u_0).
\end{aligned}
\]
The first two terms satisfy
\[
\|\sigma(u-u_0)\|_0\lesssim \iota^{1/2}\|f\|_0,
\qquad
\|\sigma(u_0)-\Pi_{cr}\sigma(u_0)\|_{L^2(\mathcal{T}_h)} 
\lesssim h|\sigma(u_0)|_1
\lesssim h\|f\|_0.
\]
The definition of $\Pi_{cr}$ and a scaling argument give
\[
\begin{aligned}
\|\Pi_{cr}\sigma(u-u_0)\|_{0,T}
&\lesssim
h_T^{1/2}\sum_{F\subset\partial T}\|\sigma(u-u_0)\|_{0,F}.
\end{aligned}
\]
The multiplicative trace inequality
\cite[Lemma~12.15]{ErnGuermond2021} yields
\[
h_T^{1/2}\|\sigma(u-u_0)\|_{0,F}
\lesssim
\|\sigma(u-u_0)\|_{0,T}
+h_T^{1/2}\|\sigma(u-u_0)\|_{0,T}^{1/2}|\sigma(u-u_0)|_{1,T}^{1/2}.
\]
Consequently,
\[
\|\Pi_{cr}\sigma(u-u_0)\|_{L^2(\mathcal{T}_h)}^2 
\lesssim
\|\sigma(u-u_0)\|_0^2+h\|\sigma(u-u_0)\|_0|\sigma(u-u_0)|_1.
\]
It follows from \eqref{Reg-u0} and \eqref{Reg} that
\[
\|\sigma(u-u_0)\|_0\lesssim \iota^{1/2}\|f\|_0,
\qquad
|\sigma(u-u_0)|_1
\lesssim \iota^{-1/2}\|f\|_0.
\]
Therefore,
\[
\|\Pi_{cr}\sigma(u-u_0)\|_{L^2(\mathcal{T}_h)} 
\lesssim
\left(\iota^{1/2}+h^{1/2}\right)\|f\|_0
\lesssim h^{1/2}\|f\|_0,
\]
where we have used $h>\iota$. Combining the preceding estimates
(and taking $h\leq1$) gives
\[
\|\sigma(u)-\Pi_{cr}\sigma(u)\|_{L^2(\mathcal{T}_h)} 
\lesssim h^{1/2}\|f\|_0.
\]
Finally, the Hessian-commuting property gives
\[
\nabla_h\sigma(u-\Pi_Mu)
=(I-P_T^0)\nabla\sigma(u). 
\]
Hence, by the $L^2$-stability and approximation property of
$P_T^0$, 
\[
\begin{aligned}
\iota\|\nabla_h\sigma(u-\Pi_Mu)\|_{L^2(\mathcal{T}_h)} 
&\lesssim
\min\left\{
\iota^{1/2}\|f\|_0,\,
\iota h^{s-2}\|\sigma(u)\|_{s-1}
\right\}.
\end{aligned}
\]
\end{proof}
Above lemma shows that boundary layer may cause the convergence rate almost up to half order when $\iota$ becomes small. And for smooth enough case, small $\iota$ leads full convergence rate for interpolation error. 
\section{Error analysis}
The error analysis of the proposed nonconforming method involves
both interpolation and consistency errors. The interpolation
estimates have already been established in the previous section. In this section,
we focus on estimating the consistency error.
Define the following mesh-dependent norm:
\begin{equation}
	\|v\|_{*,h}^2
	:=
	\interleave v\interleave_{\iota,h}^2+s(v,v),
	\qquad
	v\in H_0^2(\Omega;\mathbb R^d)+V_M^0.
	\label{eq:star-norm}
\end{equation}
The preceding interpolation estimates yield the following result.
\begin{mylem}
	\label{lem:star-norm-interpolation}
	Let $2\leq s\leq3$, and let
	\[
		u\in H_0^2(\Omega;\mathbb R^d)
		\cap H^s(\Omega;\mathbb R^d)
	\]
	be the solution of \eqref{eq:LSG-cweak}. Assume that
	\eqref{Reg-u0}--\eqref{Reg} and \eqref{eq:regularity-elastic} hold. Then
	\begin{equation}
		\|u-\Pi_Mu\|_{*,h}
		\lesssim
		\min\left\{
		h^{1/2}\|f\|_{0},
		\left(h^{s-1}+\iota h^{s-2}\right)
		\|\sigma(u)\|_{s-1}
		\right\}.
		\label{eq:star-norm-interpolation}
	\end{equation}
	The hidden constant is independent of $h$, $\iota$, and $\lambda$.
\end{mylem}
To derive an error estimate under the minimal regularity
assumption, let $P_{\omega_F}^0$ denote the $L^2$-orthogonal
projection onto the space of constants on $\omega_F$,
where $\omega_F$ is the patch consisting of the two elements
sharing $F$.
\begin{mylem}[Consistency]
	\label{lem:consistency}
	Let $u\in H_0^2(\Omega;\mathbb R^d)$ be the exact solution of
	\eqref{eq:LSG-cweak}. Then, for any $w_h\in V_M^0$,
	\begin{equation}
		\begin{aligned}
		\bigl|(f,w_h)_{\mathcal{T}_h}-\mathcal A(u,w_h)\bigr| 
		\lesssim{}&
		\eta_{1,h}(u)
		\interleave w_h\interleave_{\iota,h}
		+
		\eta_{2,h}(u)\,s(w_h,w_h)^{1/2},
		\end{aligned}
		\label{eq:consistency-estimate}
	\end{equation}
	where
	\begin{align}
		\eta_{1,h}(u):={}&
		h\|f\|_{0} 
		+\|\sigma(u)-P_T^0\sigma(u)\|_{L^2(\mathcal{T}_h)}		+ 
		\iota
		\|\nabla\sigma(u)
		-P_T^0\nabla\sigma(u)\|_{L^2(\mathcal{T}_h)} 
		\notag\\
		&+
		\left(
		\sum_{F\in\mathcal F_h^i}
		\|\sigma(u)-P_{\omega_F}^0\sigma(u)\|_{0,\omega_F}^2 
		\right)^{1/2}
				+
		\iota
		\left(
		\sum_{F\in\mathcal F_h^i}
		\|\nabla\sigma(u)
		-P_{\omega_F}^0\nabla\sigma(u)\|_{0,\omega_F}^2 
		\right)^{1/2},
        \label{eq:eta-zero}
\end{align}
and
\begin{align}
		\eta_{2,h}(u):={}&
		\lambda
		\left(
		\sum_{F\in\mathcal F_h^i}
		h_F
		\bigl\|
		[P_T^0\operatorname{div}u 
		-\operatorname{div}_h\Pi_Mu]
		\bigr\|_{0,F}^2
		\right)^{1/2}.
		\label{eq:eta-stab}
	\end{align}
\end{mylem}


\begin{proof}
	Since $I_Aw_h\in H_0^2(\Omega)$, we have
	$$
a_{\iota}(u,A_hw_h)=	\iota^2 (\nabla \sigma( u), \nabla\epsilon( A_hw_h))_{\mathcal{T}_h}+  (\sigma( u), \epsilon( A_hw_h))_{\mathcal{T}_h}=(f,A_hw_h)_0. 
	$$
This implies
\begin{equation}
(f,w_h)_{\mathcal{T}_h}-\mathcal A(u,w_h) 
=(f,z_h)_{\mathcal{T}_h}-a_\iota(u,z_h) 
   +J(u,z_h)-s(u,w_h),
\label{eq:consistency-identity}
\end{equation}
where $z_h=w_h-I_Aw_h.$
Then 
\begin{align*}
	a_\iota(u,z_h)
   -J(u,z_h)+s(u,w_h)&=  \underbrace{(\sigma( u), \epsilon_h(z_h))_{\mathcal{T}_h} 
-( \{P_T^0\sigma( u)\},  \llbracket z_h\rrbracket)_{\mathcal{F}_h}}_{(I)}\\&+	\underbrace{\iota^2 (\nabla \sigma( u), \nabla_h\epsilon_h(z_h))_{\mathcal{T}_h}}_{(II)}+ \underbrace{ \lambda^2(h_F[P_T^0\mathrm{div} u],[P_T^0\operatorname{div}_h w_h])_{\mathcal{F}^i_h}}_{(III)}. 
 \end{align*}
To avoid estimates involving higher order norms that may grow as $\iota\rightarrow 0$, we first project the stress onto elementwise constants and then integrate the projected term by parts.
 \begin{align*} 
 (I) &= \underbrace{(P_T^0\sigma(u), \epsilon_h(w_h- I_Aw_h))_{\mathcal{T}_h} 
 -( \{P_T^0\sigma( u)\},  \llbracket w_h-I_Aw_h\rrbracket)_{\mathcal{F}_h}}_{(I')}\\&+(\sigma( u)-P_T^0\sigma( u), \epsilon_h(w_h- I_Aw_h))_{\mathcal{T}_h} . 
 \end{align*}
Using integration by parts, trace and inverse inequalities for $(I')$, we have
\begin{align*}
	(I')=&([P_T^0\sigma(u)],\{w_h-I_Aw_h\})_{\mathcal{F}^i_h}=([P_T^0\sigma(u)-P_{\omega_F}^0\sigma(u)],\{w_h-I_Aw_h\})_{\mathcal{F}^i_h}\\&\lesssim \left(\sum_{F\in\mathcal{F}^i_h}\|P_T^0\sigma(u)-P_{\omega_F}^0\sigma(u)\|_{0,\omega_F}^2\right)^{1/2}|w_h|_{H^1(\mathcal{T}_h)}, 
\end{align*}
The second term is estimated as follows. 
\begin{align*}
	(II) &= \iota^2 (\nabla_h \sigma( u-\Pi_Mu), \nabla_h\epsilon_h(w_h- I_Aw_h))_{\mathcal{T}_h} + \iota^2 (\nabla_h \sigma( \Pi_Mu), \nabla_h\epsilon_h(w_h- I_Aw_h))_{\mathcal{T}_h}\\ 
\lesssim&(\iota\|\nabla\sigma(u)-P_T^0\nabla\sigma(u)\|_{L^2(\mathcal{T}_h)}   )\interleave  w_h\interleave _{\iota,h}   + \iota^2\sum_{T \in \mathcal{T}_{h}}\int_{\partial T}\nabla_h\sigma(\Pi_Mu)\pmb n_T:\epsilon_h(w_h-I_Aw_h) \mathrm{d}s. 
\end{align*}
 Then by the weak continuity property of Morley element, we have
\begin{align*}
&	\sum_{T \in \mathcal{T}_{h}}\int_{\partial T}\nabla_h\sigma(\Pi_Mu)\pmb n_T:\epsilon_h(w_h-I_Aw_h) \mathrm{d}s \\ 
&= (\{\nabla_h\sigma(\Pi_Mu)\}-P_F^0\{\nabla_h\sigma(\Pi_Mu)\},\llbracket\epsilon_h(w_h-I_Aw_h)\rrbracket)_{\mathcal{F}_h}\\&+([\nabla_h\sigma(\Pi_Mu)-P_{\omega_F}^0\nabla\sigma(u)],\{\epsilon_h(w_h-I_Aw_h)\})_{\mathcal{F}^i_h}. 
\end{align*} 
Here $P_F^0$ is the constant $L^2$ projection on $F$. Trace inequality leads that
\[
(II)
\lesssim
\iota\left(
\|\nabla\sigma(u)-P_T^0\nabla\sigma(u)\|_{L^2(\mathcal T_h)}
+
\left(
\sum_{F\in\mathcal F_h^i}
\|\nabla\sigma(u)-P_{\omega_F}^0\nabla\sigma(u)\|_{0,\omega_F}^2
\right)^{1/2}
\right)
\interleave w_h\interleave_{\iota,h}.
\]
For the last term, note that $[P_T^0\operatorname{div}_h w_h]|_F$ is a constant, we have  
\begin{align*}
(III)
={}&
\lambda^2
\sum_{F\in\mathcal F_h^i}h_F
\left(
[P_T^0\operatorname{div}u 
-\operatorname{div}_h\Pi_Mu],
[P_T^0\operatorname{div}_h w_h] 
\right)_{0,F} 
\\
&\leq
\lambda\left(
\sum_{F\in\mathcal F_h^i}
h_F
\|[P_T^0\operatorname{div}u
-\operatorname{div}_h\Pi_Mu]\|_{0,F}^2
\right)^{1/2}
s(w_h,w_h)^{1/2}.
\end{align*}
Combining $
|(f,z_h)_{\mathcal{T}_h}|\lesssim h\|f\|_0|w_h|_{H^1(\mathcal{T}_h)} 
,$ we finish the proof. 
\end{proof}
\begin{remark}
	Let
	$u\in H^s(\Omega;\mathbb R^d)\cap H_0^2(\Omega;\mathbb R^d)$,
	with $2\leq s\leq 3$, be the exact solution of \eqref{eq:LSG-cweak}. Then
	\begin{equation}
		\sup_{0\neq w_h\in V_M^0}
		\frac{
			\bigl|(f,w_h)_{\mathcal{T}_h}-\mathcal A(u,w_h)\bigr| 
		}{
			\|w_h\|_{*,h} 
		}
		\lesssim
		h\|f\|_{0} 
		+\bigl(h+\iota h^{s-2}\bigr)
		\|\sigma(u)\|_{s-1}, 
		\label{eq:consistency-regular-bound}
	\end{equation}
Moreover, \eqref{Reg-u0} and \eqref{Reg} imply
	\begin{equation}
		\sup_{0\neq w_h\in V_M^0}
		\frac{
			\bigl|(f,w_h)_{\mathcal{T}_h}-\mathcal A(u,w_h)\bigr| 
		}{
			\|w_h\|_{*,h} 
		}
		\lesssim
		\min\left\{
			h^{1/2}\|f\|_{0},\, 
			h\|f\|_{0} 
			+\bigl(h+\iota h^{s-2}\bigr)
			\|\sigma(u)\|_{s-1} 
		\right\}.
		\label{eq:consistency-robust-bound}
	\end{equation}
	The hidden constants are independent of $h$, $\iota$, and $\lambda$.
\end{remark}
\begin{mylem}[The second Strang lemma]
	\label{lem:second-strang}
	Let $u\in H_0^2(\Omega;\mathbb R^d)$ be the exact solution of
	\eqref{eq:LSG-cweak}. Then the discrete problem
	\eqref{eq:LSG-weak} admits a unique solution $u_h\in V_M^0$,
    	such that
	\begin{align}
		\|u-u_h\|_{*,h}
		\lesssim{}&
		\|u-\Pi_Mu\|_{*,h}
		+\eta_{3,h}(u)
		+
		\sup_{0\neq v_h\in V_M^0}
		\frac{
		|(f,v_h)_{\mathcal{T}_h}-\mathcal A(u,v_h)| 
		}{
		\|v_h\|_{*,h}
		}.
		\label{eq:second-strang}
	\end{align}
Here
	\begin{align}
		\eta_{3,h}(u)=
		\|\sigma(u)-\Pi_{cr}\sigma(u)\|_{L^2(\mathcal{T}_h)} 
		+		(
		\sum_{F\in\mathcal F_h^i}
		h_F^{-1}\|\{e_h\}\|_{0,F}^2
		)^{1/2}.
		\label{eq:eta-interpolation}
	\end{align}
where $e_h:=\Pi_Mu-u.$
\end{mylem}

\begin{proof}
By Lemma \ref{lem:covercivity}, the Lax--Milgram theorem
	implies that there exists a unique $u_h\in V_M^0$ such that
	\[
		\mathcal A(u_h,v_h)=(f,v_h)_{\mathcal{T}_h}, 
		\qquad\forall v_h\in V_M^0.
	\]
	By the triangle inequality,
	\[
		\|u-u_h\|_{*,h}
		\leq
		\|u-\Pi_Mu\|_{*,h}
		+\|\Pi_Mu-u_h\|_{*,h}.
	\]
	Coercivity implies
	\begin{align}
		\|w_h\|_{*,h}^2
		&=
		\interleave w_h\interleave_{\iota,h}^2
		+s(w_h,w_h)
		\lesssim
		\mathcal A(w_h,w_h)
		\notag\\
		&=
		\mathcal A(\Pi_Mu-u,w_h)
		+\mathcal A(u,w_h)-(f,w_h)_{\mathcal{T}_h}. 
		\label{eq:discrete-error-identity}
	\end{align}
Here $w_h=\Pi_Mu-u_h$. Then we estimate 
	\[
		\mathcal A(e_h,w_h)
		=
		a_\iota(e_h,w_h)
		-J(e_h,w_h)
		+J(w_h,e_h)
		+s(e_h,w_h).
	\]
	Using the commuting properties
	\[
		\sigma(\Pi_Mu)=\Pi_{cr}\sigma(u),
		\qquad
		\nabla_h\sigma(\Pi_Mu)
		=P_T^0\nabla\sigma(u), 
	\]
and the fact that $\nabla_h\epsilon_h(w_h)$ is a constant tensor, we obtain
	\begin{align}
		|a_\iota(e_h,w_h)|
		\lesssim{}&
		\|\sigma(u)
		-\Pi_{cr}\sigma(u)\|_{L^2(\mathcal{T}_h)} 
		\interleave w_h\interleave_{\iota,h}.
		\label{eq:aiota-interpolation}
	\end{align} Inverse trace inequality gives
\begin{equation}
	|J(e_h,w_h)|\lesssim 	\|\sigma(u)-\Pi_{cr}\sigma(u)\|_{L^2(\mathcal{T}_h)} 
	|w_h|_{H^1(\mathcal{T}_h)}. 
 	\label{eq:J-rho-w}   
\end{equation}
	Next we turn our attention to the term $J(w_h,e_h)$.
	Elementwise integration by parts gives
	\begin{align}
		J(w_h,e_h)
		={}&
		(P_T^0\epsilon_h(w_h), 
		 \sigma(e_h))_{\mathcal T_h}-
		\sum_{F\in\mathcal F_h^i}
		\left\langle
		[P_T^0\sigma(w_h)], 
		\{e_h\}
		\right\rangle_F .
		\label{eq:J-w-rho-identity}
	\end{align}
	The volume term satisfies
	\begin{align*}
		\left|
		(P_T^0\epsilon_h(w_h), 
		 \sigma(e_h))_{\mathcal T_h}
		\right|
		&\lesssim
		\|\sigma(e_h)\|_{L^2(\mathcal{T}_h)} 
		|w_h|_{H^1(\mathcal{T}_h)}\leq 
		\|\sigma(u)-\Pi_{cr}\sigma(u)\|_{L^2(\mathcal{T}_h)} 
		\|w_h\|_{*,h}.
	\end{align*}
	Then we only need to estimate the second term.
	Using
	\[
		[P_T^0\sigma(w_h)] 
		=
		2\mu
		[P_T^0\epsilon_h(w_h)] 
		+
		\lambda
		[P_T^0\operatorname{div}_h w_h] 
		, 
	\]
	the strain part can be bounded by
	\begin{align}
		&2\mu
	|
		\sum_{F\in\mathcal F_h^i}
		\left\langle
		[P_T^0\epsilon_h(w_h)], 
		\{e_h\}
		\right\rangle_F
		|
		\lesssim
				(
		\sum_{F\in\mathcal F_h^i}
		h_F^{-1}\|\{e_h\}\|_{0,F}^2
		)^{1/2}\|w_h\|_{*,h}.
		\label{eq:J-shear-part}
	\end{align}
 
	The divergence part satisfies
	\begin{align}
		&\lambda
		|
		\sum_{F\in\mathcal F_h^i}
		\left\langle
		[P_T^0\operatorname{div}_h w_h], 
		\{e_h\} 
		\right\rangle_F
		|
		\notag\\
		&\qquad\lesssim
		s(w_h,w_h)^{1/2}		(
		\sum_{F\in\mathcal F_h^i}
		h_F^{-1}\|\{e_h\}\|_{0,F}^2
		)^{1/2}
		\leq
				(
		\sum_{F\in\mathcal F_h^i}
		h_F^{-1}\|\{e_h\}\|_{0,F}^2
		)^{1/2}\|w_h\|_{*,h}.
		\label{eq:J-w-rho-estimate}
	\end{align}
	Similarly, the last term can be bounded by the interpolation error
	and $s(w_h,w_h)$.
	Indeed, the Cauchy--Schwarz inequality gives
	\begin{align}
		|s(e_h,w_h)|
		&\leq
		s(e_h,e_h)^{1/2}
		s(w_h,w_h)^{1/2}\leq
		s(u-\Pi_Mu,u-\Pi_Mu)^{1/2}
		\|w_h\|_{*,h}.
		\label{eq:s-interpolation-estimate}
	\end{align}
	Combining
	\eqref{eq:aiota-interpolation},
	\eqref{eq:J-rho-w},
	\eqref{eq:J-w-rho-estimate}, and
	\eqref{eq:s-interpolation-estimate}, we obtain
	\begin{align}
		|\mathcal A(e_h,w_h)|
		\lesssim
		\Bigl(
		\eta_{I,h}(u)
		+s(u-\Pi_Mu,u-\Pi_Mu)^{1/2}
		\Bigr)
		\|w_h\|_{*,h}.
		\label{eq:A-interpolation-estimate}
	\end{align}
	Since
	\[
		s(u-\Pi_Mu,u-\Pi_Mu)^{1/2}
		\leq
		\|u-\Pi_Mu\|_{*,h},
	\]
	\eqref{eq:discrete-error-identity} yields
	\begin{align*}
		\|w_h\|_{*,h}^2
		\lesssim{}&
		\Bigl(
		\eta_{3,h}(u)
		+\|u-\Pi_Mu\|_{*,h}
		\Bigr)
		\|w_h\|_{*,h}+
		\bigl|
		\mathcal A(u,w_h)-(f,w_h)_{\mathcal{T}_h} 
		\bigr|.
	\end{align*}
	Dividing by $\|w_h\|_{*,h}$, and then using the triangle
	inequality, proves \eqref{eq:second-strang}.
\end{proof}

\begin{myTheo}[Parameter-robust a priori error estimate]
	\label{thm:robust-error}
	Let $2\leq s\leq3$. Let
	$u\in H_0^2(\Omega;\mathbb R^d)\cap
	H^s(\Omega;\mathbb R^d)$ and $u_h\in V_M^0$ be the solutions of
	\eqref{eq:LSG-cweak} and \eqref{eq:LSG-weak}, respectively.
	Assume that the parameter-uniform regularity estimates
	\eqref{Reg-u0}--\eqref{Reg} hold. Then
\begin{equation}
	\|u-u_h\|_{*,h}
	\lesssim
	\min\left\{
	h^{1/2}\|f\|_0,(h(\|\sigma(u)\|_{1}+\|f\|_0)+\iota h^{s-2}\|\sigma(u)\|_{s-1})
	\right\},
\end{equation}
In particular, if $\iota\|\sigma(u)\|_2+\|\sigma(u)\|_1\lesssim \|f\|_0,$ then
\[
	\|u-u_h\|_{*,h}
	\lesssim h\|f\|_0.
\]
	The hidden constant is independent of $h$, $\iota$, and $\lambda$.
\end{myTheo}

\section{A superpenalty Morley formulation}
\label{sec:morley-superpenalty}

We also consider a symmetric Morley method obtained by adding a
superpenalty term in the second order elasticity term, following \cite{SPMWX}. This formulation does not involve projection operators. But its convergence rate and parameter robustness depend on the choice of a parameter $p$. Let $0<p\leq 1$, define
\begin{align}
j_{p,h}(w,v)
&:=\sum_{F\in\mathcal F_h}
  h_F^{-(2p+1)}
  \bigl\langle[\![w]\!],[\![v]\!]\bigr\rangle_F,
\label{sp:eq:penalty}\\
A_{\mathrm{sp},h}(w,v)
&:=a_\iota(w,v)+j_{p,h}(w,v),
\label{sp:eq:bilinear}
\end{align}
where
\[
a_\iota(w,v)
=\iota^2(\nabla_h\sigma_h(w),\nabla_h\varepsilon_h(v))_{\mathcal T_h}
 +(\sigma_h(w),\varepsilon_h(v))_{\mathcal T_h}.
\]
Here $[\![v]\!]=v^+\otimes n^++v^-\otimes n^-$ on an interior
face and $[\![v]\!]=v\otimes n$ on a boundary face. Thus the penalty
controls the full displacement jump, and the sum includes boundary faces.
The method is to find $u_h^{\mathrm{sp}}\in V_M^0$ such that
\begin{equation}
A_{\mathrm{sp},h}(u_h^{\mathrm{sp}},v_h)=(f,v_h)
\qquad\forall v_h\in V_M^0.
\label{sp:eq:method}
\end{equation}
Set
\begin{equation}
\|v\|_{\mathrm{sp},h}^2
:=a_\iota(v,v)+j_{p,h}(v,v).
\label{sp:eq:norm}
\end{equation}
The discrete Korn inequalities imply
\begin{equation}
|v_h|_{H^1(\mathcal{T}_h)}^2+\iota^2|v_h|_{H^2(\mathcal{T}_h)}^2
\lesssim\|v_h\|_{\mathrm{sp},h}^2
=A_{\mathrm{sp},h}(v_h,v_h)
\qquad\forall v_h\in V_M^0.
\label{sp:eq:coercivity}
\end{equation}
Then, we have
\begin{equation}
    \interleave u-u_h\interleave_{\iota,h}\lesssim \frac{|(f,w_h)-\mathcal{A}_{sp,h}(u,w_h)|+|\mathcal{A}_{sp,h}(u-\Pi_Mu,w_h)|}{\|w_h\|_{sp,h}} + \interleave u-\Pi_Mu \interleave_{\iota,h}.
    \label{eq:sp-error-decomp}
\end{equation}
Here $u$ is the solution of \eqref{eq:LSG-cweak} and $u_h$ is the finite element solution of \eqref{sp:eq:method}.
\begin{mylem}[Low-regularity consistency estimate]
\label{sp:low:consistency}
Let $u\in H_0^2(\Omega;\mathbb R^d)$ solve
\eqref{eq:LSG-cweak}, with $f\in L^2(\Omega;\mathbb R^d)$. Then, for every $w_h\in V_M^0$,
\begin{equation}
\begin{aligned}
&\bigl|(f,w_h)_{\mathcal T_h}
-\mathcal A_{\mathrm{sp},h}(u,w_h)\bigr|\\
&\quad\lesssim
\Bigl(\|\sigma(e_h)\|_{L^2(\mathcal T_h)}
+\iota\|\nabla_h\sigma(e_h)\|_{L^2(\mathcal T_h)}
+h\|f\|_0+h^p\|\sigma(u)\|_0\Bigr)
\|w_h\|_{\mathrm{sp},h}.
\end{aligned}
\label{sp:low:consistency-estimate}
\end{equation}
Here $e_h=\Pi_Mu-u$. The hidden constant is independent of $h$, $\iota$, and $\lambda$.
\end{mylem}

\begin{proof}
Set $z_h=w_h-I_Aw_h$. Since $u\in H_0^2(\Omega;\mathbb R^d)$,
$j_{p,h}(u,w_h)=0$,
\begin{equation}
\begin{aligned}
&(f,w_h)_{\mathcal T_h}
-\mathcal A_{\mathrm{sp},h}(u,w_h)=a_\iota(e_h,z_h)
+(f,z_h)_{\mathcal T_h}-a_\iota(\Pi_Mu,z_h).
\end{aligned}
\label{sp:low:split}
\end{equation}
By Lemma \ref{Lem:Enrich-error}, we have
\begin{equation}
\begin{aligned}
|a_\iota(e_h,z_h)|
\lesssim{}&\Bigl(\|\sigma(e_h)\|_{L^2(\mathcal T_h)}
+\iota\|\nabla_h\sigma(e_h)\|_{L^2(\mathcal T_h)}\Bigr)
\|w_h\|_{\mathrm{sp},h}.
\end{aligned}
\label{sp:low:interpolation-pairing}
\end{equation}
Since $\sigma(\Pi_Mu)$ is affine on each element,
\begin{equation}
\begin{aligned}
a_\iota(\Pi_Mu,z_h)
={}&-(\operatorname{div}_h\sigma(\Pi_Mu),z_h)_{\mathcal T_h}+\sum_{T\in\mathcal T_h}
\bigl\langle\sigma(\Pi_Mu)n_T,z_h\bigr\rangle_{\partial T}\\
&+\iota^2\sum_{T\in\mathcal T_h}
\bigl\langle(\nabla_h\sigma(\Pi_Mu))n_T,
\epsilon_h(z_h)\bigr\rangle_{\partial T}.
\end{aligned}
\label{sp:low:polynomial-ibp}
\end{equation}
By the weak continuity property of the Morley element, we have 
\begin{equation}
\begin{aligned}
&(f,z_h)_{\mathcal T_h}-a_\iota(\Pi_Mu,z_h)=\sum_{T\in\mathcal T_h}
(f+\operatorname{div}_h\sigma(\Pi_Mu),z_h)_{0,T}\\
&\qquad-\sum_{F\in\mathcal F_h^i}
\bigl\langle[\sigma(\Pi_Mu)],\{z_h\}\bigr\rangle_F-\iota^2\sum_{F\in\mathcal F_h^i}
\bigl\langle[\nabla_h\sigma(\Pi_Mu)],
\{\nabla z_h\}\bigr\rangle_F\\
&\qquad-\sum_{F\in\mathcal F_h}
\bigl\langle\{\sigma(\Pi_Mu)\},
\llbracket w_h\rrbracket\bigr\rangle_F.
\end{aligned}
\label{sp:low:residual-ibp}
\end{equation}
Following the standard bubble
construction used in \cite{SPMWX}, we obtain
\begingroup
\small
\begin{align}
\|f+\operatorname{div}_h\sigma(\Pi_Mu)\|_{0,T}
&\lesssim
\iota^2h_T^{-2}\|\nabla_h\sigma(e_h)\|_{0,T}
+h_T^{-1}\|\sigma(e_h)\|_{0,T}
+\|f\|_{0,T},
\label{sp:low:cell-bubble}\\
\iota^2\|[\nabla_h\sigma(\Pi_Mu)]n_F\|_{0,F}
&\lesssim
\sum_{T\subset\omega_F}\!
\bigl(
\iota^2h_T^{-1/2}\|\nabla_h\sigma(e_h)\|_{0,T}
+h_T^{1/2}\|\sigma(e_h)\|_{0,T}
+h_T^{3/2}\|f\|_{0,T}
\bigr),
\label{sp:low:moment-bubble}\\
\|[\sigma(\Pi_Mu)]\|_{0,F}
&\lesssim
\sum_{T\subset\omega_F}\!
\bigl(
\iota^2h_T^{-3/2}\|\nabla_h\sigma(e_h)\|_{0,T}
+h_T^{-1/2}\|\sigma(e_h)\|_{0,T}
+h_T^{1/2}\|f\|_{0,T}
\bigr).
\label{sp:low:traction-bubble}
\end{align}
\endgroup
For the last term, the polynomial inverse trace inequality gives
\begin{align}
&\left|\sum_{F\in\mathcal F_h}
\bigl\langle\{\sigma(\Pi_Mu)\},
\llbracket w_h\rrbracket\bigr\rangle_F\right|
\notag\leq
\left(\sum_{F\in\mathcal F_h}h_F^{2p+1}
\|\{\sigma(\Pi_Mu)\}\|_{0,F}^2\right)^{1/2}
j_{p,h}(w_h,w_h)^{1/2}
\notag\\
&\quad\lesssim
h^p\|\sigma(\Pi_Mu)\|_{L^2(\mathcal T_h)}
\|w_h\|_{\mathrm{sp},h}
\notag\leq
h^p\Bigl(\|\sigma(u)\|_0
+\|\sigma(e_h)\|_{L^2(\mathcal T_h)}\Bigr)
\|w_h\|_{\mathrm{sp},h}.
\label{sp:low:average-traction}
\end{align}
Since $h^p\leq1$, combining these bounds with
\eqref{sp:low:split} and \eqref{sp:low:interpolation-pairing}
proves \eqref{sp:low:consistency-estimate}.
\end{proof}
From Lemma \ref{sp:low:consistency}, the choice of $p$ is important for the convergence rate of the consistency term.
Increasing $p$ will improve the convergence rate of the consistency term. However the term 
$|\mathcal{A}_{sp,h}(e_h,w_h)|$ require taking a balance between the regularity of $u$ and $p$.

By the commuting property \eqref{eq:commuting} and the fact that
$\nabla_h\epsilon_h(w_h)$ is elementwise constant, we have
\[
(\nabla_h\sigma(e_h),\nabla_h\epsilon_h(w_h))_{\mathcal T_h}=0.
\]
Since $\llbracket u\rrbracket=0$, it follows that
\[
\mathcal A_{\mathrm{sp},h}(e_h,w_h)
=
(\sigma(e_h),\epsilon_h(w_h))_{\mathcal T_h}
+j_{p,h}(\Pi_Mu,w_h).
\]
The scaled trace and Poincaré inequalities imply
\[
\|\llbracket\Pi_Mu\rrbracket\|_{0,F}
\lesssim
\sum_{T\subset\omega_F}h_T^{1/2}|e_h|_{1,T}.
\]
Consequently, the Cauchy--Schwarz inequality yields
\[
|j_{p,h}(\Pi_Mu,w_h)|
\lesssim
\left(
\sum_{T\in\mathcal T_h}h_T^{-2p}|e_h|_{1,T}^{2}
\right)^{1/2}
j_{p,h}(w_h,w_h)^{1/2}.
\]
Combining these estimates, we obtain
\begin{equation}
|\mathcal A_{\mathrm{sp},h}(e_h,w_h)|
\lesssim
\left(
\|\sigma(e_h)\|_{L^2(\mathcal T_h)}
+
\left(
\sum_{T\in\mathcal T_h}h_T^{-2p}|e_h|_{1,T}^{2}
\right)^{1/2}
\right)
\|w_h\|_{\mathrm{sp},h}.
\label{sp:interpolation-term}
\end{equation}
If, in addition, $u\in H^s(\Omega;\mathbb R^d)$ with
$2\leq s\leq3$, the interpolation estimates imply
\begin{equation}
|\mathcal A_{\mathrm{sp},h}(e_h,w_h)|
\lesssim
\bigl(
h^{s-1}\|\sigma(u)\|_{s-1}
+h^{s-p-1}|u|_s
\bigr)
\|w_h\|_{\mathrm{sp},h}.
\label{sp:interpolation-term-rate}
\end{equation}
Combining Lemma \ref{sp:low:consistency}, \eqref{eq:sp-error-decomp} and \eqref{sp:interpolation-term-rate} we obtain the following error estimate theorem.  
\begin{myTheo}[Error estimate for the superpenalty method]
\label{sp:error-Hs}
Let $0<p\leq1$ and $h\leq1$. Suppose that
$u\in H_0^2(\Omega;\mathbb R^d)\cap H^s(\Omega;\mathbb R^d)$,
with $2\leq s\leq3$, solves \eqref{eq:LSG-cweak}, and let
$u_h\in V_M^0$ be the solution of \eqref{sp:eq:method}.
Then
\begin{equation}
\begin{aligned}
|u-u^{sp}_h|_\iota
\lesssim{}&
\bigl(h^{s-1}+\iota h^{s-2}\bigr)
\|\sigma(u)\|_{s-1}
+h^{s-p-1}|u|_s\\
&+h\|f\|_0+h^p\|\sigma(u)\|_0.
\end{aligned}
\label{sp:error-Hs-estimate}
\end{equation}
The hidden constant is independent of $h$, $\iota$, and $\lambda$.
\end{myTheo}
Theorem~\ref{sp:error-Hs} exhibits a balance between the factors
$\iota h^{s-2}$, $h^{s-p-1}$, and $h^p$, together with the
corresponding regularity norms of the exact solution.
For a prescribed regularity index $2<s\leq3$, the last two
factors can be balanced by choosing $p=(s-1)/2$.
In particular, for $p=1$, the penalty term becomes
\[
j_{1,h}(w,v)
=\sum_{F\in\mathcal F_h}h_F^{-3}
\bigl\langle\llbracket w\rrbracket,
\llbracket v\rrbracket\bigr\rangle_F.
\]
If $u\in H_0^2(\Omega;\mathbb R^d)\cap H^3(\Omega;\mathbb R^d)$,
the interpolation and consistency estimates give
\begin{equation}
\interleave u-u^{sp}_h\interleave_{\iota,h}
\lesssim
h\bigl(
\|f\|_0+\|\sigma(u)\|_1
+\iota\|\sigma(u)\|_2+|u|_3
\bigr).
\label{sp:eq:first-order}
\end{equation}
Under the assumed regularity estimates, this bound is uniform
in $\lambda$ for fixed $\iota$. However, the solution norms
need not remain bounded as $\iota\to0$.
Thus, \eqref{sp:eq:first-order} does not by itself provide
a convergence rate uniform in $\iota$.

Under the parameter-uniform regularity estimates
\eqref{Reg-u0}--\eqref{Reg} and \eqref{eq:regularity-elastic},
the decomposition $u=(u-u_0)+u_0$ and a multiplicative Morley
interpolation estimate yield, from the low-regularity error bound,
\[
\interleave u-u^{sp}_h\interleave_{\iota,h}
\lesssim
\bigl(h^p+h^{1/2-p}\bigr)\|f\|_0,
\qquad 0<p<\tfrac12.
\]
Choosing $p=1/4$ therefore gives a quarter-order convergence
rate uniform in both $\iota$ and $\lambda$.

The superpenalty scheme is symmetric and does not involve
projection operators in its formulation. Its choice of $p$
should account for the available regularity and the desired
parameter robustness. In the presence of boundary layers,
the best uniform rate obtained by the above analysis is
one quarter; a uniform half order rate is not established.
\section{Numerical Experiments}
In this section, we present several numerical experiments to
assess the robustness of the scheme \eqref{eq:LSG-weak} with
respect to the parameters $\iota$ and $\lambda$.
The first experiment uses a sufficiently smooth exact solution
to verify the expected convergence rates and validate the
implementation. The second experiment tests the convergence
of the scheme for a solution exhibiting boundary layers. In all numerical experiments in this section, the convergence rates are computed from the errors on the two finest meshes.

\subsection{Verification}
In this section we select the smooth enough exact solution as 
\[
\text{2D: } u_{\iota}=\mathrm{curl}(\sin^3\pi x_1\sin^3\pi x_2),\quad\text{3D: }u_{\iota}=\mathrm{curl}(\Pi_{i=1}^3\sin^3x_i,\Pi_{i=1}^3\sin^3x_i,\Pi_{i=1}^3\sin^3x_i)^T,
\]
on the domain $\Omega = (0,1)^d,~d=2,3.$
\begin{figure}[htbp]
  \centering
  \includegraphics[width=0.35\linewidth]{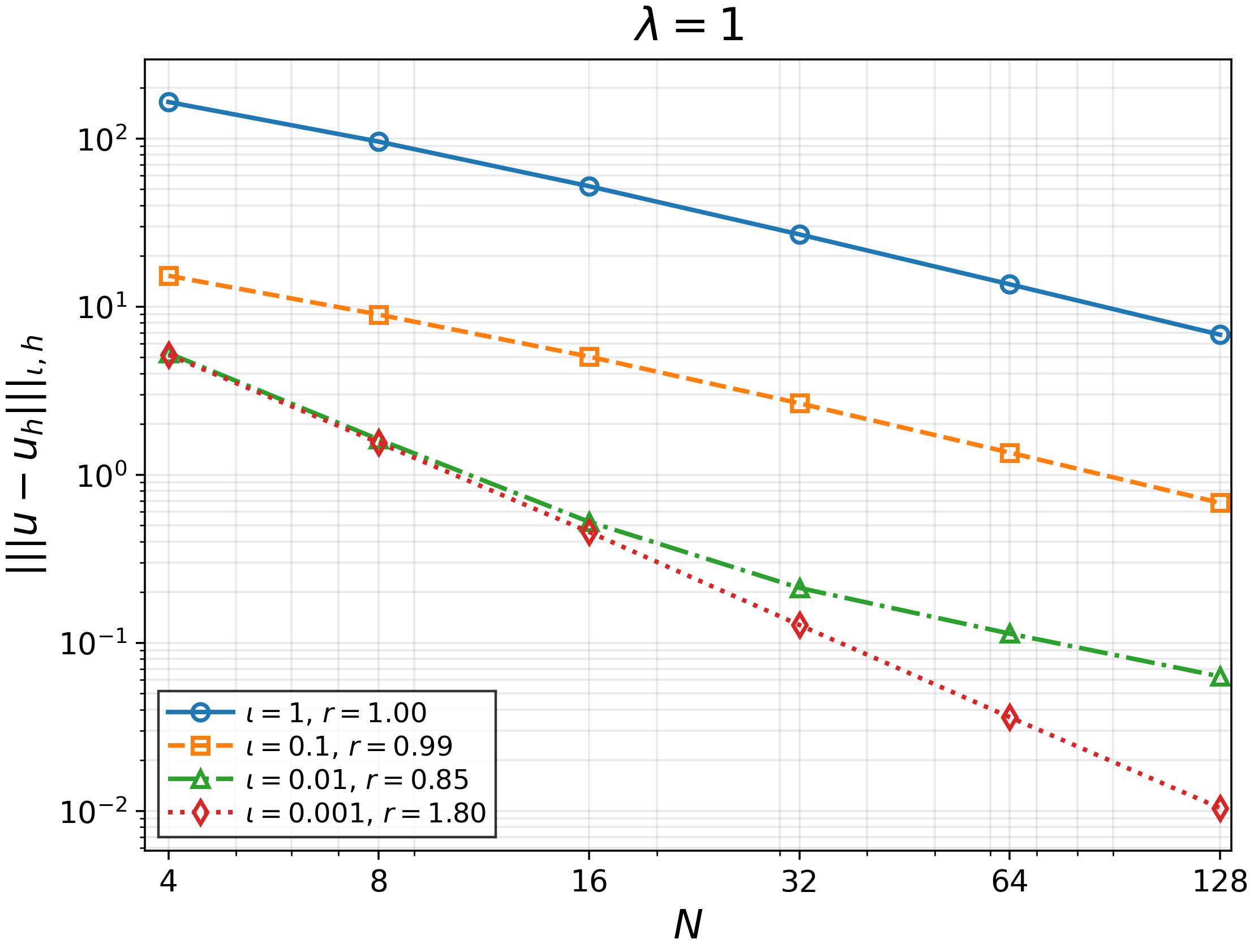}\qquad
    \includegraphics[width=0.35\linewidth]{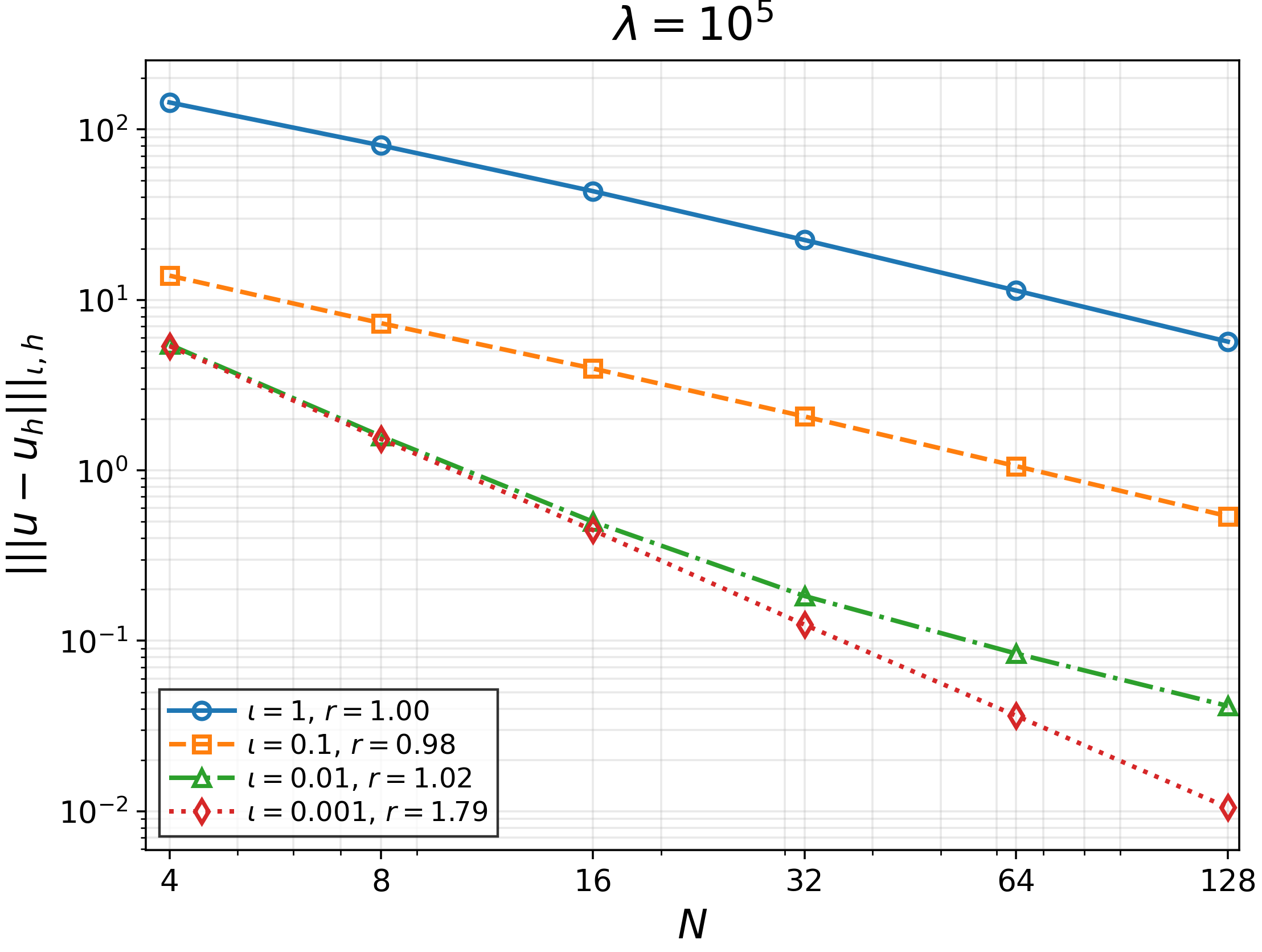}
  \caption{ Example 1 in 2D: convergence rate in energy norm $\interleave\cdot\interleave$.}
  \label{fig:smooth}
\end{figure}
\begin{figure}[htbp]
  \centering
  \includegraphics[width=0.35\linewidth]{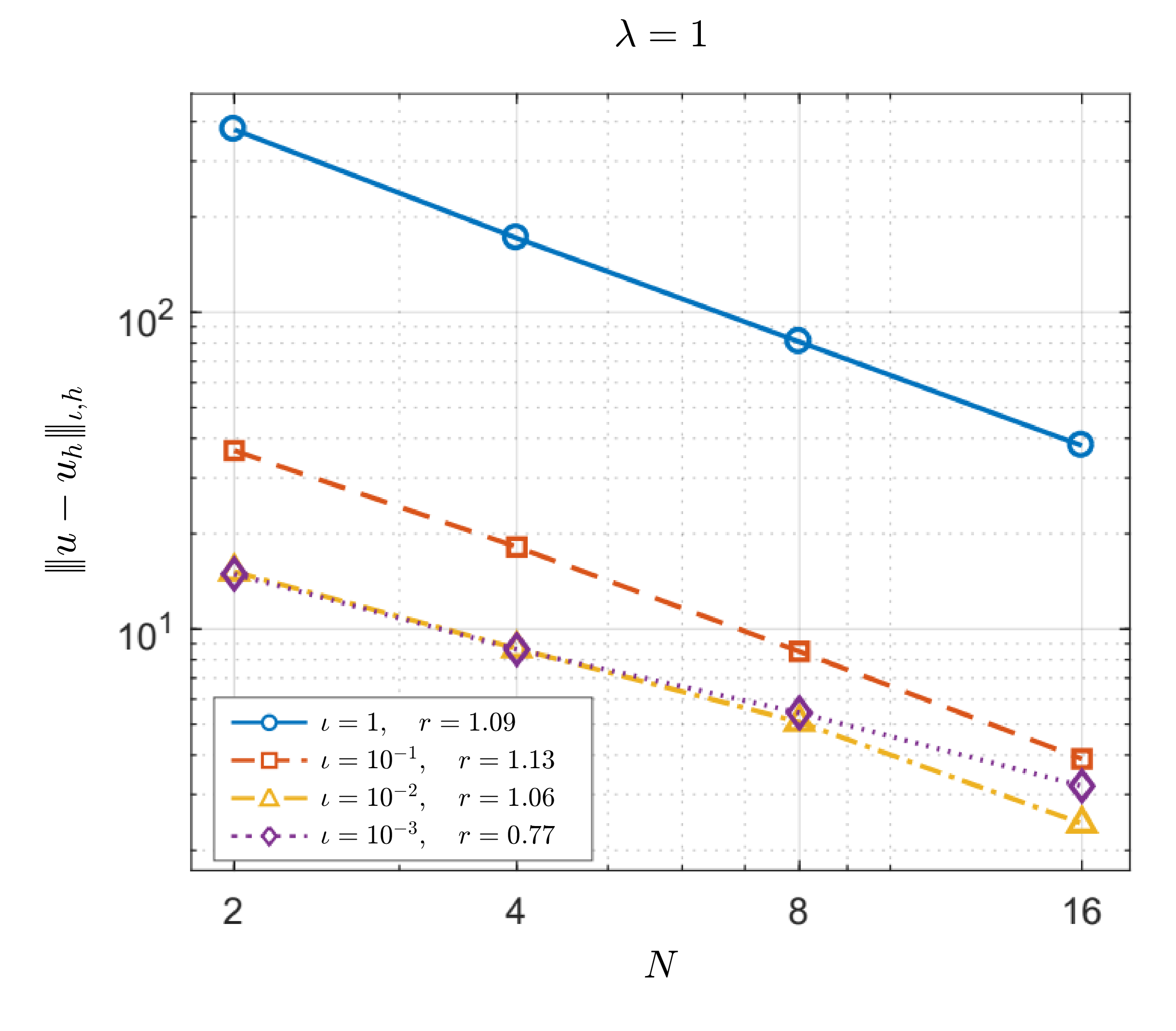}\quad
    \includegraphics[width=0.35\linewidth]{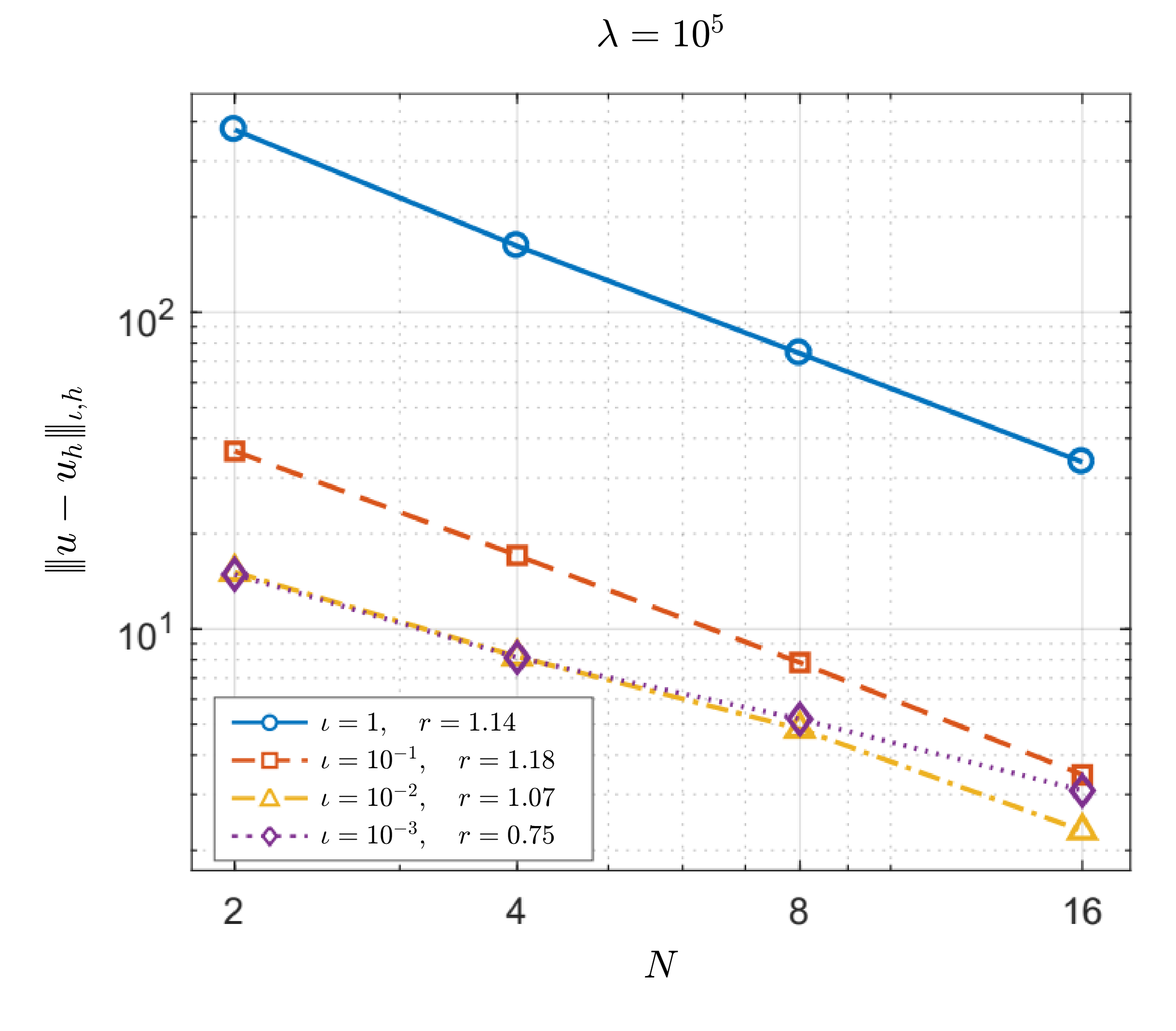}
  \caption{ Example 1 in 3D: convergence rate in energy norm $\interleave\cdot\interleave$.}
  \label{fig:smooth3D}
\end{figure}
We consider $\lambda=1$ and $10^5$, with
$\iota=1,10^{-1},10^{-2},10^{-3}$.
Figure~\ref{fig:smooth} and \ref{fig:smooth3D} presents the
energy-norm errors on uniformly refined meshes.
These results are almost consistent with the theoretical first order
error estimate for smooth solutions, which imply \eqref{eq:LSG-cweak} is robust with $\lambda$ and $\iota$. For $\iota=10^{-3}$,
a higher convergence rate is observed on the tested meshes. 
\subsection{Boundary layer test}
Let $\Omega=(0,1)^d,~d=2,3$ and $\mu=1$. In this subsection, we 
select the exact solution as
\[
u_\iota=
\begin{cases}
\operatorname{curl}\psi_\iota, & d=2,\\
\operatorname{curl}(\psi_\iota\boldsymbol{e}), & d=3,
\end{cases}
\qquad
\boldsymbol{e}=(1,1,1)^{\mathsf T},
\]
where $\pmb x=(x_1,\cdots,x_d),~d=2,3,$ and 
\[
\begin{aligned}
q_\iota(t)=\sin^2(\pi t)-c_\iota\Bigl[
\iota^2\bigl(
e^{-t/\iota}+e^{-(1-t)/\iota}-1-\eta_\iota
\bigr)
+\iota(1-\eta_\iota)t(1-t)
\Bigr],
\end{aligned}
\]
with
\[
\eta_\iota=e^{-1/\iota},
\quad
c_\iota=
\frac{2\pi^2}
{1+\eta_\iota-2\iota(1-\eta_\iota)}.
\]
Here $0<\iota\leq 1.$
Since $q_\iota=q_\iota'=q_\iota''=0$ at $t=0,1$,
we have $u_\iota\in H_0^2(\Omega;\mathbb{R}^d)$.
By directly computation the corresponding right hand side $\|f_{\iota}\|_0$ is bounded and $|u_\iota|_2\sim O(\iota^{-1/2})$. 
\begin{figure}[htbp]
  \centering
  \includegraphics[width=0.35\linewidth]{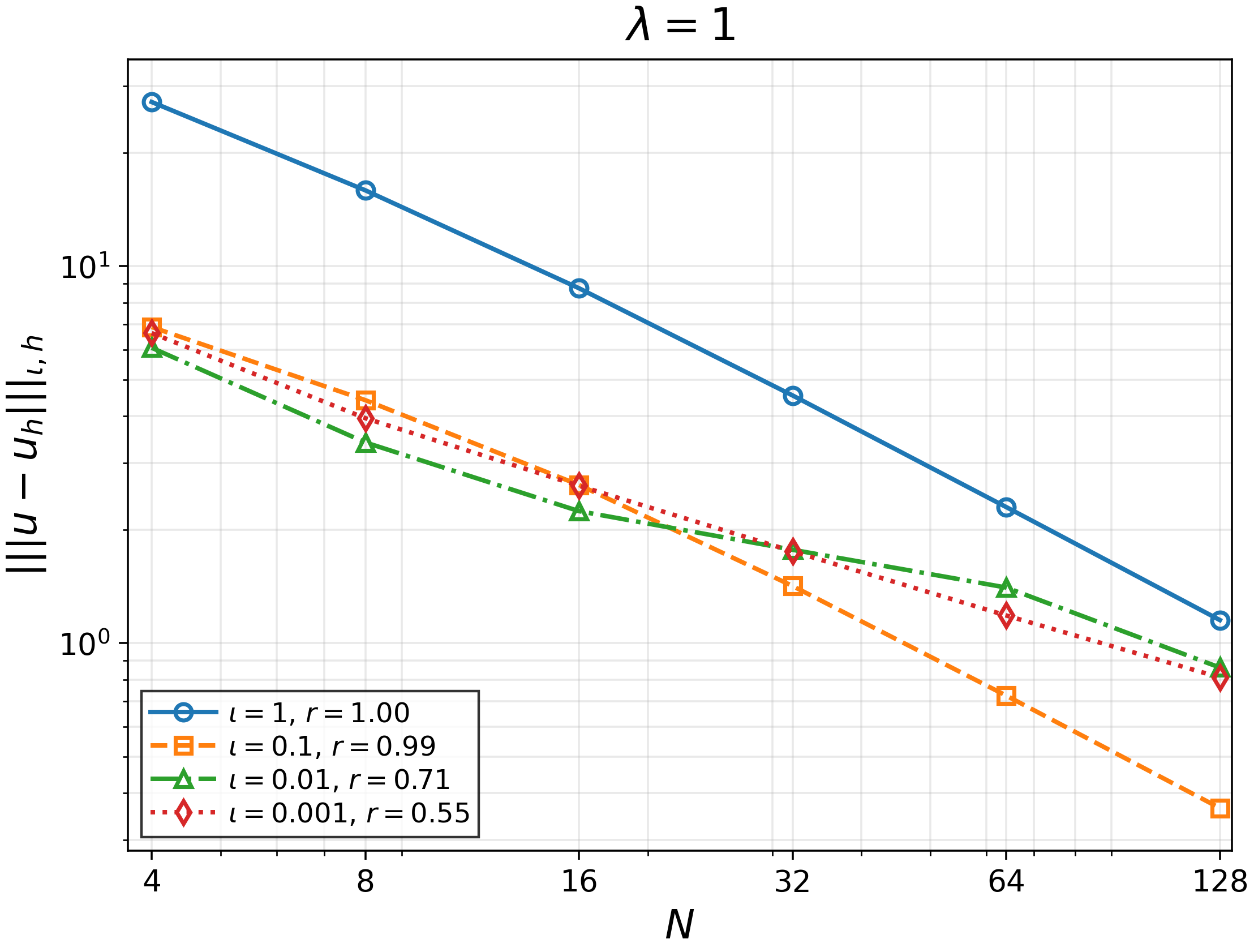}\qquad
    \includegraphics[width=0.35\linewidth]{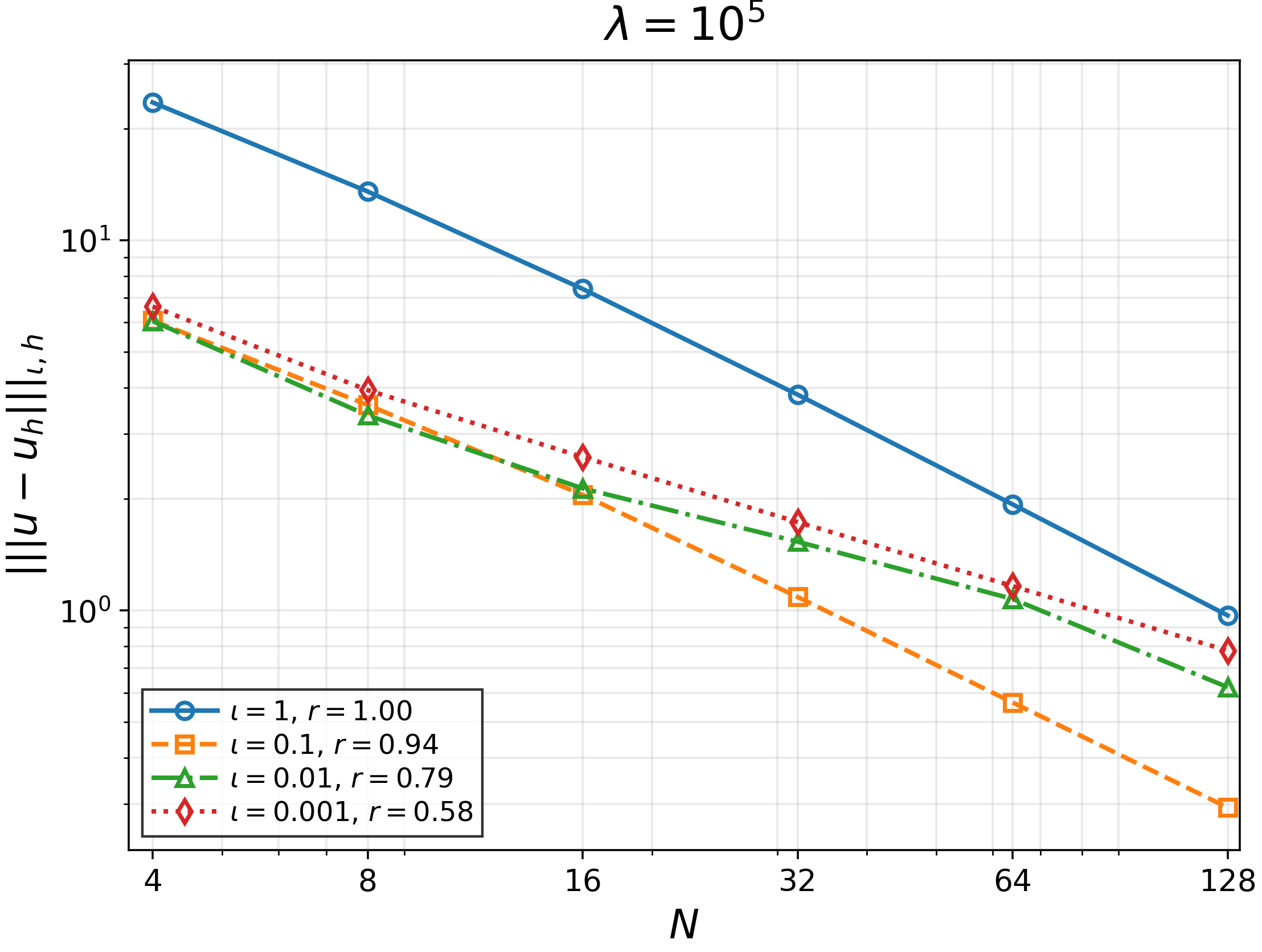}
  \caption{ Example 2 in 2D: convergence rate in energy norm $\interleave\cdot\interleave$.}
  \label{fig:layer}
\end{figure}

\begin{figure}
  \centering
  \includegraphics[width=0.35\linewidth]{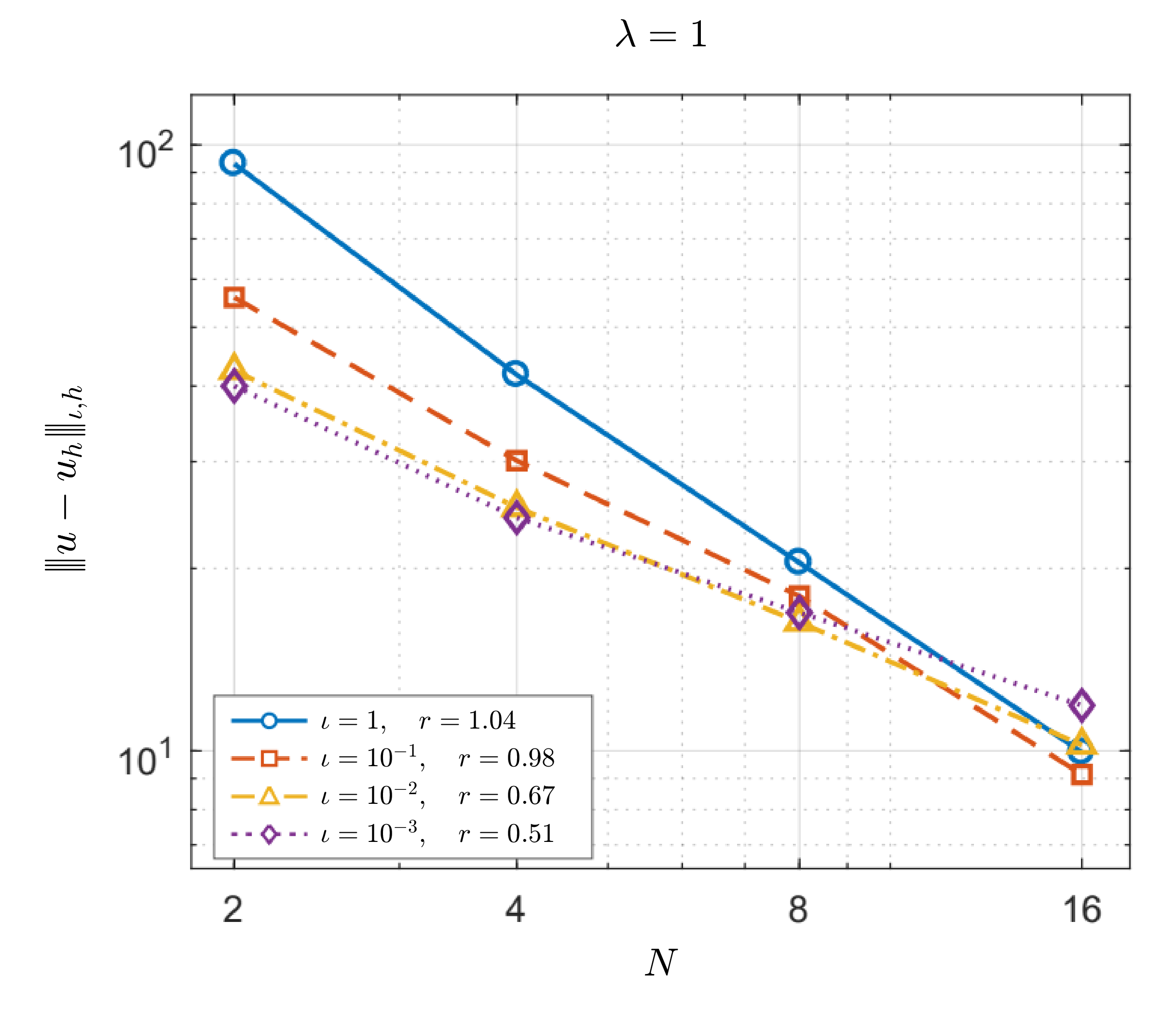}\quad
    \includegraphics[width=0.35\linewidth]{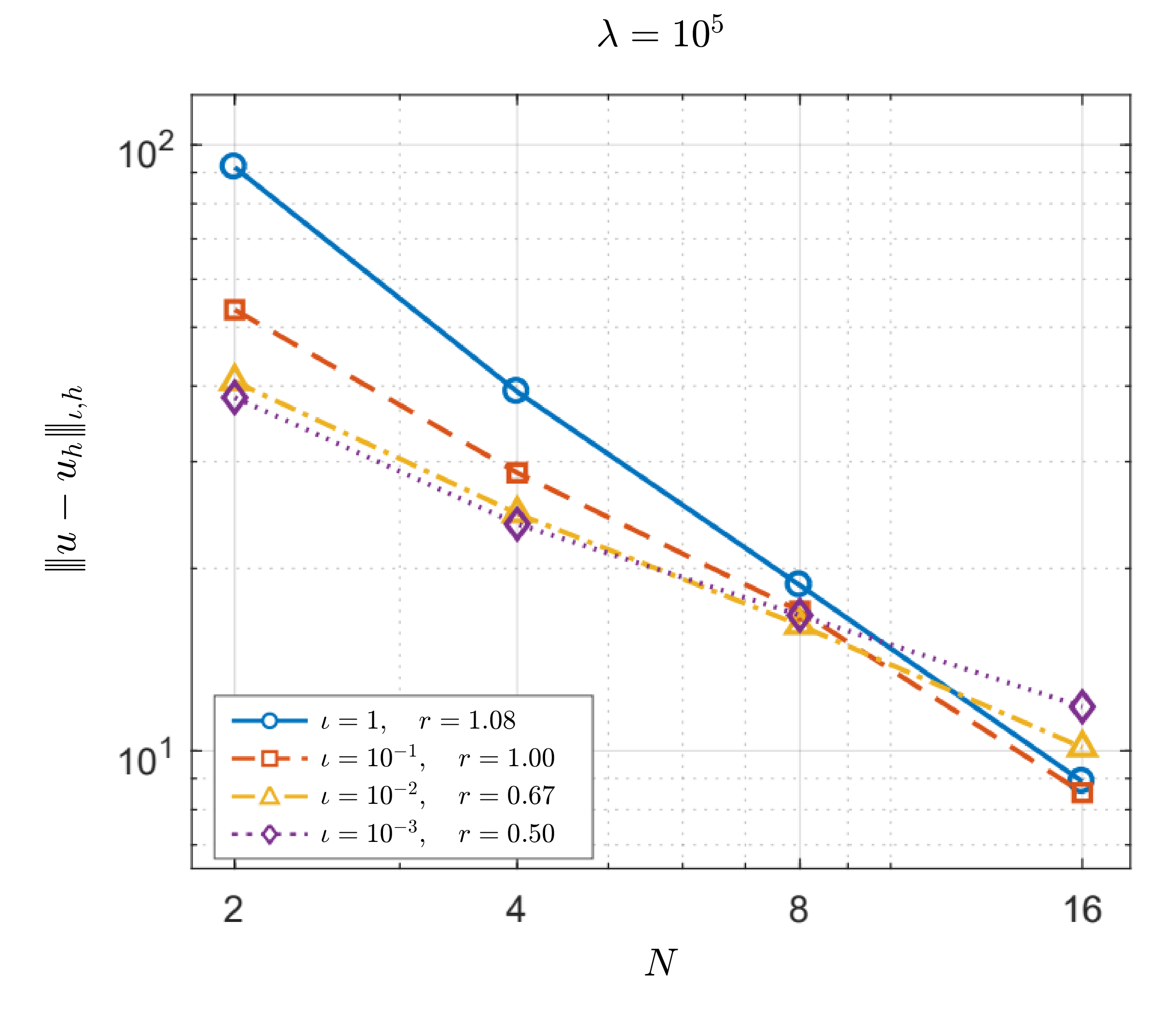}
  \caption{ Example 2 in 3D: convergence rate in energy norm $\interleave\cdot\interleave$.}
  \label{fig:layer3D}
\end{figure}

As in Example~1, we test $\iota=1,10^{-1},10^{-2},10^{-3}$
and $\lambda=1,10^5$ for the exact solution with boundary layer case.
Figure \ref{fig:layer} and \ref{fig:layer3D} show that as $\iota$ decreases, the observed convergence rate in the energy norm decreases from first order to approximately half order.

\bibliographystyle{plain}
\bibliography{ref}

\end{document}